%% file: main.tex
\documentclass{article}
\usepackage{amsmath,amssymb,amsthm}
\usepackage{tikz}
\usepackage{enumerate}
\usepackage{url}

\newtheorem{theorem}{Theorem}

\newtheorem{cor}[theorem]{Corollary}
\newtheorem{lemma}[theorem]{Lemma}
\newtheorem{prop}[theorem]{Proposition}

\theoremstyle{definition}
\newtheorem{definition}[theorem]{Definition}

\newtheorem{remark}[theorem]{Remark}

\numberwithin{theorem}{section}
\numberwithin{equation}{section}
\numberwithin{figure}{section}

\newtheorem*{theorem*}{Theorem}
\newtheorem*{conj*}{Conjecture}

\renewcommand{\phi}{\varphi}
\newcommand{\allset}[1]{\mathcal{J}_{#1}}
\newcommand{\maxset}[1]{\mathcal{M}_{#1}}
\newcommand{\maxweightset}[1]{\mathcal{M}^*_{#1}}
\newcommand{\maxweight}[1]{w_{#1}^*}

\newcommand{\permset}[1]{\mathcal{S}_{#1}}

\newcommand{\lehmerset}[1]{\mathcal{L}_{#1}}
\newcommand{\blockset}[1]{\mathcal{B}_{#1}}
\newcommand{\mwblockset}[1]{\mathcal{B}^*_{#1}}
\newcommand{\extblockset}[1]{\overline{\mathcal{B}}^*_{#1}}

\newcommand{\row}[1]{\rho_{#1}}
\newcommand{\ind}[1]{\lambda_{#1}}

\DeclareMathOperator{\inv}{inv}
\DeclareMathOperator{\maxinv}{maxinv}
\DeclareMathOperator{\mininv}{mininv}

\def\T{T}

\def\pattern{\underline{32}1}

\newcounter{x}
\newcounter{y}

\newcommand\emptycol[1]{
   \draw[thick] (#1, 0) -- (#1, #1) -- (#1-1, #1) -- (#1-1, 0) -- cycle;   
   \foreach \c in {1,...,#1} {
      \draw[thick] (#1, \c) -- (#1-1, \c);
  }      
}

\newcommand\fillcol[3]{
    \fill[#3] (#1, 0) -- (#1, #2) -- (#1-1, #2) -- (#1-1, 0) -- cycle;
    
}

\newcommand\tri[3][(0,0)]{
 \setcounter{y}{0}
  \foreach \a in {#2} {
      \addtocounter{y}{1}
    }

\begin{scope}[shift={#1}]
\begin{scope}[scale=0.33, shift={(-\value{y}/2-1/2,0)}]    
 \setcounter{x}{0}
  \foreach \a in {#2} {
      \addtocounter{x}{1}
      \fillcol{\value{x}}{\a}{#3}
      \emptycol{\value{x}}
    }

  \end{scope}
\end{scope}
}

\allowdisplaybreaks

\begin{document}
\title{$\pattern$-Avoiding Permutations with Maximum Inversion Number}
\author{Andrew Beveridge\footnote{Department of Mathematics, Statistics and Computer Science, Macalester College, St Paul, MN, USA}, Kristin Heysse$^*$ and Paige Robertson\footnote{Department of Mathematics, University of Colorado, Boulder, CO, USA}}

\date{}

\maketitle

\begin{abstract}
A permutation $\pi \in \permset{n}$ is $\pattern$-avoiding when there do not exist $1 \leq i < i+1 < j \leq n$ such that $\pi_i > \pi_{i+1} > \pi_j$. We determine the maximum inversion number for $\pattern$-avoiding permutations and count the number of permutations that achieve this maximum. We then provide a direct construction that enumerates these permutations.
\end{abstract}


\medskip


\input{intro-alt}

\input{jump-float}

\input{max-weight}

%%%%%%

\input{results}

%%%%%

\input{conclusion}

%%% BIBLIOGRAPHY %%% 

%%\nocite{*}

\bibliographystyle{plain}

%\bibliography{bibliography}

\bibliography{biblio}

\end{document}

%% file: intro-alt.tex
\section{Introduction}
\label{sec:intro}

The analysis of pattern containment and avoidance within permutations is a cornerstone of enumerative combinatorics. Given permutations $\sigma = (\sigma_1, \sigma_2, \ldots, \sigma_k) \in \permset{k}$ and $\pi = (\pi_1, \pi_2, \cdots, \pi_n) \in \permset{n}$, we say that $\pi$ \emph{contains} $\sigma$ when there exist indices $1 \leq i_1 < i_2 < \cdots < i_k \leq n$ such that the entries of subsequence $(\pi_{i_1}, \pi_{i_2}, \ldots, \pi_{i_k})$ are in the same relative order as the entries of $\sigma$; otherwise, we say that $\pi$ \emph{avoids} $\sigma$. 
Babson and Steingr\'imsson \cite{Babson2000} initiated the study of \emph{vincular patterns} (also called \emph{generalized patterns}), where we further require that some $\pi_{i_j}$ and $\pi_{i_{j+1}}$ are adjacent in $\pi$.
See Kitaev \cite{kitaev} for a broad introduction to patterns in permutations, and see Steingr\'imsson \cite{steingrimsson} for a survey on vincular patterns. 

\begin{definition}
A \emph{vincular pattern} $\sigma$ is a permutation in $\permset{k}$, some of whose consecutive letters are underlined. If $\pi \in \permset{n}$ contains vincular pattern $\sigma$, and $\sigma$ contains $\underline{\sigma_i \sigma_{i+1} \cdots \sigma_j}$, then the letters corresponding to $\sigma_i, \sigma_{i+1}, \ldots, \sigma_j$ in $\pi$ must be adjacent. The collection of $n$-permutations that avoid vincular pattern $\sigma$ is denoted $\permset{n}(\sigma)$.
\end{definition}

For example, $14352$ has a $\pattern$ pattern, while $51423$ does not (though it does have two 321 patterns). 

Permutation families that avoid vincular patterns have received widespread attention. Claesson \cite{claesson} gave enumerative formulas for families of permutations avoiding length three vincular patterns with one pair of adjacent letters.
%: 8 of these 12 families are counted by the Bell numbers, and the remaining 4 are counted by the Catalan numbers. 
Biers-Ariel \cite{biers-ariel} enumerated permutations avoiding certain length four vincular patterns. 
Mansour and Shattuck \cite{mansour2025-flatten} explored pattern avoidance for flattened derangements.
Various researchers have studied vincular pattern avoidance on cyclic permutations \cite{domagalski_2022,li2022,mansour2022}.
More generally, Dimitrov \cite{dimitrov2023} explored 
pattern avoidance  with designated gap sizes
between pairs of consecutive letters, and vincular patterns for other combinatorial structures have also been studied \cite{lin_vincular_2020,mansour2025}. 

%Herein, we focus our attention on $\permset{n}(\pattern)$, the family of $\pattern$-avoiding permutations.

We prove an extremal result about the inversion number of $\pattern$-avoiding permutations.

% \begin{definition}
%     A permutation $\pi \in \permset{n}$ has a $\pattern$ pattern when there exist $1 \leq i < i+1 < j \leq n$ such that $\pi_i > \pi_{i+1} > \pi_j$. The set of $\pattern$-avoiding permutations of $[n]$ is denoted $\permset{n}(\pattern)$.
% \end{definition}

\begin{definition}
The \emph{inversion number} of permutation $\pi \in \permset{n}$ is  
$$
\inv(\pi) = \{ (i,j) : i < j \mbox{ and } \pi_i > \pi_j \}.
$$   
\end{definition}

The inversion number $\inv(\pi)$ is also the number of interchanges of consecutive elements needed to convert permutation $\pi$ into the identity permutation $123 \cdots n$.
Inversion numbers for (regular) pattern avoiding permutations have received ample attention.
Barcucci et al.~\cite{barcucci} explored these permutations using multivariate generating functions. Frankl\'in \cite{franklin} developed enumerative formulas by inversion number alone. Chen et al.~\cite{chen-mei-wang} counted 321-avoiding permutations by length and inversion number. Cheng et al.~\cite{cheng_inversion_2013} studied the inversion polynomial for 321-avoiding permutations.

% More recently, Claesson et al.~\cite{Claesson2023} 

Given vincular pattern $\sigma$, 
let $\maxinv \permset{n}(\sigma)$ 
%$= \max_{\pi \in \permset{n}(\scriptsize{\pattern)}} \inv(\pi)$ 
denote the maximum inversion number for $\sigma$-avoiding permutations of $[n]$. 
Let $\mathcal{Z}_n(\sigma) \subset \permset{n}(\sigma)$ denote the subset of permutations achieving $\maxinv \permset{n}(\sigma)$. Our main theorem fully characterizes $\mathcal{Z}_n(\pattern)$.

\begin{theorem}
    \label{thrm:correct-perm}
    For $n \geq 1$, the number of $\pattern$-avoiding permutations of $[n+1]$ with maximum inversion number is  
    \[
    |\mathcal{Z}_{n+1}(\pattern)| = 
    \binom{\row{n}}{\ind{n}}+\binom{\row{n}+1}{\ind{n}-1}
    \]
    where 
$$
\row{n} = \left\lfloor \frac{-3 + \sqrt{8n+17}}{2} \right\rfloor
\qquad
\mbox{and}
\quad
\ind{n} = n + 2 - \T_{\row{n}+1}
$$    
and $\T_i = {i+1 \choose 2} = \sum_{j=1}^i j$ is the $i$th triangular number. 
    These permutations have inversion number
    \[
    \maxinv \permset{n+1}(\pattern)=\T_{n-\row{n}} + \sum_{i=1}^{\row{n}}  \T_i =  {n-\row{n} +1 \choose 2} + { \rho_n + 2 \choose 3}.
    \]
\end{theorem}

The sequence 
\begin{equation}
\label{eqn:count-seq}
 \{ |\mathcal{Z}_{n+1}(\pattern)| \}_{n \geq 1} = 1 , 2, 2, 1, 3, 4, 3, 1, 4, 7, 7, 4, 1, 5, 11, 14, 11, 5, \ldots   
\end{equation}
matches sequence A209561  and the sequence
\begin{equation}
\label{eqn:weight-seq}
 \{ \maxinv \permset{n+1}(\pattern)  \}_{n \geq 1} = 1, 2, 4, 7, 10, 14, 19, 25, 31, 38, 46, 55, 65, 75, 86, 98, 111, 125, \ldots 
\end{equation}
matches sequence A023536 in the OIES \cite{oeis}.

%The interpretations and formulas of Theorem~\ref{thrm:correct-perm} are novel to the OEIS.

Theorem~\ref{thrm:correct-perm} has an immediate corollary that extends this result to three other vincular patterns.
Let $\mininv \permset{n}(\sigma)$ 
%$= \max_{\pi \in \permset{n}(\scriptsize{\pattern)}} \inv(\pi)$ 
denote the minimum inversion number for $\sigma$-avoiding permutations of $[n]$. 
Let $\mathcal{Y}_n(\sigma) \subset \permset{n}(\sigma)$ denote the subset of permutations achieving $\mininv \permset{n}(\sigma)$.

\begin{cor} 
\label{cor:correct-perm}
We have
$$
|\mathcal{Z}_{n+1}(3\underline{21})| = |\mathcal{Y}_{n+1}(1\underline{23})| = |\mathcal{Y}_{n+1}(\underline{12}3)|
=
\binom{\row{n}}{\ind{n}}+\binom{\row{n}+1}{\ind{n}-1}
$$
and
$$
    \maxinv \permset{n+1}(3\underline{21})= 
    %\T_{n-\row{n}} 
    {n-\row{n} +1 \choose 2} +  { \rho_n + 2 \choose 3}
$$
while
$$
    \mininv \permset{n+1}(1\underline{23}) = \mininv \permset{n+1}(\underline{12}3) = {n+1 \choose 2} - { \rho_n + 2 \choose 3} - 
    %\T_{n-\row{n}}.
    {n-\row{n} +1 \choose 2}.
$$
\end{cor}

% \begin{proof}
% Given a permutation $\pi = \pi_1 \pi_2 \cdots \pi_{n+1} \in \permset{n+1}$, its
% \emph{reverse} is $r(\pi) = \pi_{n+1} \pi_{n} \cdots \pi_1$ and its complement is $c(\pi) = \pi_1' \pi_2' \cdots \pi_{n+1}'$ where $\pi_k' = n+1 - \pi_k$. It is clear that $\pi \in \permset{n+1}(\pattern)$ if and only if $r(\pi) \in \permset{n+1}(1\underline{23})$ and $c \circ r(\pi) \in \permset{n+1}(3\underline{21})$ and $r \circ c \circ r(\pi) \in \permset{n+1}(\underline{12}3)$. Furthermore, $\inv(\pi) = \inv(c \circ r(\pi))$ while
% $\inv(r(\pi)) = \inv(r \circ c \circ r(\pi)) = {n+1 \choose 2} - \inv(\pi)$.
% The three statements of the corollary follow directly from the theorem, since
% $\pi$ achieves $\maxinv \permset{n+1}(\pattern)$ if and only if $r(\pi)$, $c \circ r(\pi)$ and $r \circ c \circ r(\pi)$ achieve 
% $\mininv \permset{n+1}(1\underline{23})$, $\maxinv \permset{n+1}(3\underline{21})$ and $\mininv \permset{n+1}(\underline{12}3)$, respectively.
% \end{proof}

\begin{proof}
Given a permutation $\pi = \pi_1 \pi_2 \cdots \pi_{n+1} \in \permset{n+1}$, its
\emph{reverse} is $\pi^r = \pi_{n+1} \pi_{n} \cdots \pi_1$ and its complement is $\pi^c = \pi_1' \pi_2' \cdots \pi_{n+1}'$ where $\pi_k' = n+1 - \pi_k$. It is clear that $\pi \in \permset{n+1}(\pattern)$ if and only if $\pi^r \in \permset{n+1}(1\underline{23})$  and $\pi^{c} \in \permset{n+1}(\underline{12}3)$ and $\pi^{r.c} \in \permset{n+1}(3\underline{21})$. Furthermore, $\inv(\pi) = \inv(\pi^{r.c})$ while
$\inv(\pi^r) = \inv(\pi^{c}) = {n+1 \choose 2} - \inv(\pi)$.
The three statements of the corollary follow directly from the theorem, since
$\pi$ achieves $\maxinv \permset{n+1}(\pattern)$ if and only if $\pi^r$, $\pi^{c}$ and $\pi^{r.c}$ achieve 
$\mininv \permset{n+1}(1\underline{23})$, $\mininv \permset{n+1}(\underline{12}3)$ and $\maxinv \permset{n+1}(3\underline{21})$, respectively.
\end{proof}

\begin{remark}
Our main theorem is stated for $\mathcal{Z}_{n+1}(\pattern)$ rather than $\mathcal{Z}_{n}(\pattern)$. The reason will become clear in the next section, when we biject $\mathcal{S}_{n+1}(\pattern)$ to the family $\allset{n}$ of \emph{jump-float sequences} of length $n$. Moreover, the meaning behind the numbers $\row{n}$ and $\ind{n}$ (and their relationship to the triangle numbers $\T_i$) will be spelled out.
\end{remark}

\begin{remark}
Determining the maximum inversion number for 321-avoiding permutations is straightforward. A permutation $\pi \in \permset{n}$ has a 321 pattern when there exist $1 \leq i < j < k \leq n$ such that $\pi_i > \pi_j > \pi_k$. It is well-known that $\pi$ avoids the 321 pattern if and only if its elements can be partitioned into (at most) two strictly increasing subsequences. When these subsequences have lengths $m$ and $n-m$, the maximum number of possible inversions between them is $m(n-m)$, and this value is achieved by the permutation
$$
(m+1, m+2, \ldots, n, 1, 2, \ldots, m).
$$
The optimal choice is to split the permutation as evenly as possible, and this achieves the maximum inversion number $\lfloor n/2 \rfloor \cdot \lceil n/2 \rceil$. When $n$ is even, there is a unique maximizing permutation; when $n$ is odd, there are two. 
\end{remark}

\subsection{Roadmap}

The structure of the paper is as follows.
In Section~\ref{sec:jump-float}, we introduce the \emph{jump-float sequence}, which is a type of inversion sequence for $\pattern$-avoiding permutations. The section culminates in the statement of Theorems~\ref{thrm:unique} and \ref{thrm:correct}, which are a rephrasing of Theorem~\ref{thrm:correct-perm} in terms of jump-float sequences. 
Section~\ref{sec:max-weight} characterizes the structure of maximum weight jump-float sequences, which correspond to $\pattern$-avoiding permutations with maximum inversion number. In particular, we introduce the \emph{block sequence} as an alternative representation of a maximum weight jump-float sequence. 
Section~\ref{sec:results} starts with the proofs of Theorems~\ref{thrm:unique} and~\ref{thrm:correct}. We then verify the OEIS entries for sequences \eqref{eqn:count-seq} and \eqref{eqn:weight-seq}. Finally, we state and prove Theorem~\ref{thrm:structure}, which gives a direct construction method for the block sequences of 
$\mathcal{Z}_{n+1}(\pattern)$, illuminating why the size of this family is counted by the sum of the binomial coefficients $\binom{\row{n}}{\ind{n}}+\binom{\row{n}+1}{\ind{n}-1}$. We reflect on our results in Section~\ref{sec:conclusion}.

%% file: jump-float.tex
\section{Jump-float sequences}
\label{sec:jump-float}

In order to prove our main theorem, we need a representation of $\pattern$-avoiding permutations that more prominently features the inversion number. With that in mind, we define the family of jump-float sequences and show their equivalence to $\pattern$-avoiding permutations. This family admits a natural poset structure, and we characterize the maximal elements. We then translate our main result, Theorem~\ref{thrm:correct-perm}, in terms of jump-float sequences with maximum weight. 

\subsection{Jump-float sequences and $\pattern$-avoiding permutations}

To better highlight the inversion number of a permutation, we define the family of jump-float sequences and give a simple bijection to $\pattern$-avoiding permutations.

\begin{definition} A sequence $s=(s_1,\ldots s_n)$  is a \emph{jump-float sequence} when
\begin{enumerate}
    \item[(a)] $0 \leq s_i\leq i$ for $1 \leq i \leq n$.
    %\item If $s_{i} > 0$ then $s_{i+1} \leq s_{i}$ for $1 \leq i \leq n-1$.
    \item[(b)] If $s_{i} > s_{i-1}$ then $s_{i-1}=0$ for $2 \leq i \leq n$.
\end{enumerate}
The set of jump-float sequences of length $n$ is denoted $\allset{n}$. The \emph{weight} of  jump-float sequence $s$ is
$$
w(s) = \sum_{i=1}^n s_i.
$$
\label{def:jf-sequence}
\end{definition}
% We note that
% an immediate consequence of condition (b) is the following: if $i < k$ and $s_i < s_k$, then there exists $i < j < k$ such that $s_j=0$.
% Turning to some examples, the fifteen jump-float sequences of $\allset{3}$ are
% $$
% \begin{array}{ccccccccc}
% 000 & 001 & 002 & 003 & 010 & 011 & 020 & 021 & 022 \\
% 100 & 101 & 102 & 103 & 110 & 111.
% \end{array}
% $$
 The annotated example 
$$
(\underbrace{0,0}_{\mbox{\tiny{float}}}, \overset{\mbox{\tiny{jump}}}{3},\underbrace{3,2,2,0,0}_{\mbox{\tiny{float}}},
\overset{\mbox{\tiny{jump}}}{5}, \underset{\mbox{\tiny{float}}}{0},
\overset{\mbox{\tiny{jump}}}{7},\underbrace{4,4,1,1,1,0}_{\mbox{\tiny{float}}},
\overset{\mbox{\tiny{jump}}}{9},\underbrace{8,7}_{\mbox{\tiny{float}}}) \in \allset{20}
$$
illustrates the ``jump-float'' name. A ``jump'' is an ascent from $s_{i-1}=0$ to $s_{i}>0$. After each jump, we have a ``float'' subsequence of weakly decreasing entries, which must reach 0 before the next jump. 
A jump-float sequence $s=(s_1, s_2, \ldots,s_n)$ has a natural geometric representation as columns of (filled) boxes in a triangular grid, where column $i$ contains $s_i \leq i$ boxes, see Figure~\ref{fig:triangle example}. This triangle representation will be a useful visualization of  constructions in the sections that follow.

\begin{figure}[h!]
    \centering
    \begin{tikzpicture}[scale=.6]

    \begin{scope}

        \tri[(0,0)]{1,1,0,3,0,0,6,4}{violet!50};
        \node (l1) at (0,-.75) {$s=(1,1,0,3,0,0,6,4)$};
        \node (l1a) at (0,-1.75) {$\phi(s)=581273694$};
        %\node (l1a) at (0,-1.75) {$\phi(s)=(5,8,1,2,7,3,6,9,4)$};

    \end{scope}

    \begin{scope}[shift={(8,0)}]
        \tri[(0,0)]{0,2,0,4,4,3,1,0}{violet!50};
        \node (l2) at (0,-.75) {$t=(0,2,0,4,4,3,1,0)$};
        \node (l2a) at (0,-1.75) {$\phi(t)=136892745$}; 
        %\node (l2a) at (0,-1.75) {$\phi(t)=(1,3,6,8,9,2,7,4,5)$}; 
    \end{scope}

    \begin{scope}[shift={(16,0)}]
        \tri[(0,0)]{1,0,3,3,0,6,6,6}{violet!50};
        \node (l2) at (0,-.75) {$p=(1,0,3,3,0,6,6,6)$};
        \node (l2a) at (0,-1.75) {$\phi(p)=789156243$}; 
        %\node (l2a) at (0,-1.75) {$\phi(p)=(7,8,9,1,5,6,2,4,3)$}; 
    \end{scope}
    
    \end{tikzpicture}
    \caption{Jump-float sequences $s,t,p \in \allset{8}$ with their triangle representations and their corresponding $\pattern$-avoiding permutations $\phi(s), \phi(t), \phi(p) \in \permset{9}(\pattern)$.}
    \label{fig:triangle example}
\end{figure}

The bijection between jump-float sequences $\allset{n}$ and $\pattern$-avoiding permutations of $[n+1]$ is straightforward. In fact, $\allset{n}$ is a minor variation on the Lehmer code for a permutation \cite{lehmer}.

\begin{definition}
    The collection of Lehmer codes of length $n$ is 
    $$
    \lehmerset{n} = \{ (\ell_1, \ell_2, \ldots, \ell_n) : 0 \leq \ell_i \leq n-i \text{ for } 1 \leq i \leq n \}.
    $$
    % $$
    % = \{0, 1, \ldots, n-1 \} \times \{0, 1, \ldots, n-2 \} 
    % \times \cdots \times \{0, 1, 2 \} \times \{0,1\} \times \{0\}.
    % $$
    For a permutation $\pi \in S_n$, its Lehmer code is
$$
L(\pi) := ( \ell_1, \ell_2,\ldots , \ell_n) \quad \mbox{where} \quad
\ell_i = \left| \left\{ j > i : \pi_j < \pi_i \right\} \right|,
$$     
% $$
% L(\pi) := ( L(\pi)_1, \ldots , L(\pi)_n) \quad \mbox{where} \quad
% L(\pi)_i = \left| \left\{ j > i : \pi_j < \pi_i \right\} \right|,
% $$ 
and this mapping from $\permset{n}$ to $\lehmerset{n}$ is a bijection.

\end{definition}

For example, the Lehmer code of $ \pi = 51423$ is $L(\pi) = (4,0,2,0,0)$. Note that the $i$th entry counts the number of \emph{inversions} for $\pi_i$, namely the number of entries $\pi_j$ to the right of $\pi_i$ that are smaller than $\pi_i$. The sum $\sum_{i=1}^n \ell_i = \inv(\pi)$ is the inversion number of the permutation $\pi$. 
Lehmer codes and other inversion sequences, have been extensively studied. For example, researchers have explored the poset structure of Lehmer codes \cite{bouvel, denoncourt, tomie}, as well as pattern  avoidance in inversion sequences themselves \cite{mansour-2015, corteel, martinez-2016, lin_vincular_2020, testart}. We now connect  jump-float sequences to Lehmer codes of $\pattern$-avoiding permutations.

\begin{lemma}
\label{lem:jump-float-bijection}
Jump-float sequences $\allset{n}$ are in bijection with $\pattern$-avoiding permutations $\permset{n+1}(\pattern)$. Furthermore, the weight of a jump-float sequence equals the inversion number of its corresponding permutation.
\end{lemma}

\begin{proof}
We define a bijection $\phi: \allset{n} \rightarrow \permset{n+1}(\pattern)$.
Given a jump-float sequence $s=(s_1, s_2, \ldots, s_n) \in \allset{n}$, let \begin{equation}
\label{eqn:lehmer}
   r = (r_1, r_2, \ldots, r_n, r_{n+1}) :=(s_n, s_{n-1}, \ldots,s_1, 0) \in \lehmerset{n+1} 
\end{equation}
be a Lehmer code with corresponding $[n+1]$-permutation $\pi = (\pi_1, \pi_2, \ldots, \pi_n, \pi_{n+1})$. In other words, $r_i = s_{n+1-i}$ for $1 \leq i \leq n$ and $r_{n+1}=0$. The conditions of Definition~\ref{def:jf-sequence} become (a) $0 \leq r_i \leq n+1-i$, and (b) if $r_{j} > r_{j+1}$ then $r_{j+1}=0$ for $1 \leq j \leq n-1$.

For $1 \leq j \leq n$, observe that $\pi_j > \pi_{j+1}$ if and only if $r_j > r_{j+1}$. 
If $r$ was obtained from a jump-float sequence $s$, then $r_j > r_{j+1}$ forces $r_{j+1}=0$, which means that $\pi_k > \pi_{j+1}$ for all $k > j+1$.
So $j$ cannot be the start of a $\pattern$ pattern.

On the other hand, consider $\pi$ that is $\pattern$-avoiding. Suppose that $\pi_j > \pi_{j+1}$. It must be the case that $\pi_{j+1} > \pi_k$ for all $k > j+1$ (otherwise there is a $k > j$ so that $(j,j+1,k)$ form a $\pattern$ pattern). But this means that $r_{j+1}=0$ in the Lehmer code.

The equality of the weight $w(s)$ and the inversion number $\inv(\pi)$ for is clear.
\end{proof}

See Figure~\ref{fig:triangle example} for three examples of the bijective mapping $\phi: \allset{8} \rightarrow \permset{9}(\pattern)$.

\begin{cor}
   We have $| \allset{n}| = B_{n+1}$, the $(n+1)$th Bell number.
\end{cor}

\begin{proof}
Claesson (Propositions 1 and 2 of \cite{claesson}) showed that $|\permset{n}(\pattern)|=B_n$, the $n$th Bell number, and hence $| \allset{n}| = B_{n+1}$.     
\end{proof}

Further, there is a natural bijection between jump-float sequences and partitions of $[n+1]$, where zeros in the sequence correspond to the start of a new block in the partition. 

\begin{remark}
The mapping \eqref{eqn:lehmer} between jump-float sequences and Lehmer codes for $\pattern$-avoiding permutations is elementary. However, this sequence reversal and the omission of the Lehmer code's final zero greatly simplifies the indexing in our results and the readability of the arguments. %So this simple transformation is crucial.
\end{remark}

\subsection{Maximal jump-float sequences}

Having defined the family of jump-float sequences, we notice they admit a natural partial ordering. As such, we describe here the poset structure of $\allset{n}$. We also observe that the family of $\pattern$-avoiding permutations with maximum inversion number corresponds to the family of jump-float sequences with maximum weight.

\begin{lemma}
    For $s,t \in \allset{n}$ we write $s\preceq t$ when $s_i\leq t_i$ for all $i$. The relation $\preceq$ defines a partial ordering on $\allset{n}$, and the resulting poset is a meet-semilattice.
\end{lemma}

\begin{proof}
The proof that $(\allset{n}, \preceq)$ is a poset (reflexivity, antisymmetry, transitivity) is elementary.

We show that the poset is a meet-semilattice, meaning that any pair of elements has a greatest lower bound. Given jump-float sequences $s=(s_1, \ldots, s_n)$ and $t=(t_1,\ldots,t_n)$, we claim that
$$ s \wedge t := (\min \{ s_1, t_1 \}, \min \{ s_2, t_2 \}, \ldots, \min \{ s_n, t_n \})$$ 
is a jump-float sequence. Indeed, if $\min \{ s_i, t_i \} > \min \{ s_{i+1}, t_{i+1} \}$ then either $s_i > s_{i+1}$ (so $s_{i+1}=0$) or $t_i > t_{i+1}$ (so $t_{i+1}= 0$). It is clear that any $u \in \allset{n}$ such that
$u \preceq s$ and $u \preceq t$ also satisfies $u \preceq s \wedge t$.
%However the poset is not a join-semilattice. Indeed, there is no jump-float sequence that succeeds than both $(0,\ldots,0, n)$ and $(0,\ldots, 0, n-1,0)$, so they do not have a least upper bound.
\end{proof}

\begin{figure}

\begin{center}
\begin{tikzpicture}[every node/.style={font=\small}, scale=.6]

\node (000) at (0,0) {$(0,0,0)$};

\begin{scope}[shift={(0,2)}]
    
\node (100) at (-2,0) {$(1,0,0)$};
\node (010) at (0,0) {$(0,1,0)$};
\node (001) at (2,0) {$(0,0,1)$};
\end{scope}

\draw (000) -- (100);
\draw (000) -- (010);
\draw (000) -- (001);

\begin{scope}[shift={(0,4)}]

\node (110) at (-4,0) {$(1,1,0)$};    
\node (101) at (-2,0) {$(1,0,1)$};
\node (020) at (0,0) {$(0,2,0)$};
\node (011) at (2,0) {$(0,1,1)$};
\node (002) at (4,0) {$(0,0,2)$};
\end{scope}

\draw (100) -- (110);
\draw (100) -- (101);
\draw (010) -- (110);
\draw (010) -- (011);
\draw (010) -- (020);
\draw (001) -- (101);
\draw (001) -- (011);
\draw (001) -- (002);

\begin{scope}[shift={(0,6)}]

\node (111) at (-4,0) {$(1,1,1)$};    
\node (102) at (-1,0) {$(1,0,2)$};
\node (021) at (1,0) {$(0,2,1)$};
\node (003) at (4,0) {$(0,0,3)$};

\end{scope}

\draw (110) -- (111);
\draw (101) -- (111);
\draw (011) -- (111);

\draw (101) -- (102);
\draw (002) -- (102);

\draw (020) -- (021);
\draw (011) -- (021);
\draw (002) -- (003);

\begin{scope}[shift={(0,8)}]

\node (103) at (-1,0) {$(1,0,3)$};
\node (022) at (1,0) {$(0,2,2)$};

\end{scope}

\draw (102) -- (103);
\draw (003) -- (103);
\draw (021) -- (022);
\draw (002) -- (022);
    
\end{tikzpicture}
\end{center}

\caption{The meet-semilattice $(\allset{3}, \preceq)$.}
\label{fig:jump-float-3}
\end{figure}

Figure~\ref{fig:jump-float-3} shows the meet-semilattice $(\allset{3}, \preceq)$. The unique minimal element is $(0,0,0)$. The three maximal elements are $(1,1,1)$, $(1,0,3)$ and $(0,2,2)$; note that only two of them achieve the maximum weight 4. Our next lemma characterizes the maximal sequences of $\allset{n}$.

\begin{lemma}
\label{lemma:max-skip-sequence}
A jump-float sequence $s = (s_1, s_2, \ldots, s_n)$ is maximal in $(\allset{n}, \preceq)$ if and only if
the following conditions hold:
    \begin{enumerate}[(a)]
        \item We have $s_n>0$.
        \item If $s_{i-1}>0$, then $s_i \in \{0,s_{i-1}\}$ for $2\leq i \leq n$.
        \item If $s_{i-1}=0$, then $s_{i}=i$ for $2 \leq i \leq n$.
    \end{enumerate}
\end{lemma}

\begin{proof}
We prove the contrapositive statement. Let $s=(s_1, s_2, \ldots, s_n) \in \allset{n}$. We show that if $s$ violates one of these conditions then it is not maximal in $(\allset{n}, \preceq)$.    

(a) If $s_n=0$, then let $s' = (s_1, s_2, \ldots, s_{n-1}, \max \{ s_{n-1}, 1\})$. Then $s \prec s'$, so $s$ is not maximal in $\allset{n}$.

(b) Suppose that $s_{i-1} > 0$ and $s_i \notin \{ 0, s_{i-1} \}$. We have $1 \leq s_i < s_{i-1}$, so let $s' = (s_1, \ldots, s_{i-1}, s_i +1, s_{i+1}, \ldots, s_n)$. Clearly $s \prec s'$, so $s$ is not maximal in $\allset{n}$.

(c) Suppose that $s_{i-1}=0$ and $s_i < i$. Let $s' = (s_1, \ldots, s_{i-1}, s_i +1, s_{i+1}, \ldots, s_n)$.  We have $s \prec s'$, so $s$ is not maximal in $\allset{n}$.
\end{proof}

\begin{definition}
    The set of maximal jump-float sequences of length $n$ is denoted $\maxset{n}$. For convenience, we define 
    $\maxset{0} = \{ \emptyset \}.$    
\end{definition}

For $1 \leq n \leq 4$, the set $\maxset{n}$ of maximal jump-float sequences is
\begin{align*}
\maxset{1} &= \{ (1) \} \\
\maxset{2} &= \{ (0,2), (1,1)   \} \\
\maxset{3} &= \{   (0,2,2) , (1,0,3), (1,1,1) \} \\
\maxset{4} &= \{   (0,2,0,4), (0,2,2,2) , (1,0,3,3) ,  (1,1,0,4), (1,1,1,1)  \}. 
%\maxset{5} &= \{  (0,2,0,4,4), (0,2,2,0,5),  (0,2,2,2,2), (1,0,3,0,5), (1,0,3,3,3), \\
%& \qquad (1,1,0,4,4),  (1,1,1,0,5),   (1,1,1,1,1), \}.
\end{align*}
Counting all the maximal jump-float sequences is easy: they are enumerated by the Fibonacci numbers, $F_1=1$, $F_2=1$ and $F_n = F_{n-1} + F_{n-2}$ for $n \geq 3$.

\begin{lemma}
The number of maximal jump-float sequences is $|\maxset{n}|=F_{n+1}$.
\end{lemma}

\begin{proof}
% Clearly,  $M_0=1$ and $M_1=1$. For $n \geq 2$, we partition $\mathcal{M}_n$ according to whether its penultimate entry is zero or nonzero.

% Let $s = (s_1, s_2, \ldots, s_n) \in \mathcal{M}_n$. If $s_{n-1} >0$, then $s_n = s_{n-1}$, and $(s_1, s_2, \ldots, s_{n-1}) \in \mathcal{M}_{n-1}$.
% On the other hand, if $s_{n-1}=0$, then $s_n=n$ and $(s_1, s_2, \ldots, s_{n-2}) \in \maxset{n-2}$. This proves the recurrence $|\maxset{n}| = |\maxset{n-1}| + |\maxset{n-2}|$.
Clearly,  $M_1=1$ and $M_2=2$. For $n \geq 3$, we partition $\mathcal{M}_n$ according to the value of the last entry of the sequence.

Let $s = (s_1, s_2, \ldots, s_n) \in \mathcal{M}_n$. If $1 \leq s_n \leq n-1$, then $s_n = s_{n-1}$, and $(s_1, s_2, \ldots, s_{n-1}) \in \mathcal{M}_{n-1}$.
On the other hand, if $s_{n}=n$, then $s_{n-1}=0$ and $(s_1, s_2, \ldots, s_{n-2}) \in \maxset{n-2}$. This proves the recurrence $|\maxset{n}| = |\maxset{n-1}| + |\maxset{n-2}|$.
\end{proof}

%%%%% This repeats the above lemma
%%%% Keeping it in case we want to use this inductive argument instead
% Furthermore, the sequences in $\maxset{n}$ have a very specific structure. \todo{Added this lemma because I was craving something that explicitly lays out what a maximal element looks like (setting the stage for block sequences).}

% \begin{lemma}
% If $m \in \maxset{n}$ then 
% \begin{enumerate}[(a)]
%     \item $m_1 \in \{0,1\}$;
%     \item for $2 \leq i \leq n$, if $m_{i-1}=0$ then $m_i=i$, and if $m_{i-1} > 0$ then $m_i = m_{i-1}$;
%     \item $m_n > 0$.
% \end{enumerate}

% \end{lemma}

% \begin{proof}
% We give a proof by strong induction on $n$. The base case $n=1$ is clear: (a) and (c) require that $m_1=1$. Considering $m=(m_1, m_2, \ldots, m_n)$, let $j$ be the largest index such that $m_j=0$. If no such index exists, then $m=(1,1,\ldots, 1)$. Otherwise, the claim holds by induction for $(m_1, m_2, \ldots, m_{j-1})$. We have $m_j=0$, and then
% the unique optimal choice for $(m_{j+1}, m_{j+2}, \ldots, m_n)$ is $(j+1, j+1, \ldots, j+1)$.
% \end{proof}

Observe that for $n \geq 3$ there are maximal jump-float sequences that do not achieve the maximum weight.  Therefore, we need to define one more subset of $\allset{n}$.

\begin{definition}
\label{def:weight}
The maximum weight for jump-float sequences in $\allset{n}$ is \[    w_n^* = \max_{s \in \allset{n}} w(s) = \max_{s \in \maxset{n}} w(s). \]
    The collection of \emph{maximum weight jump-float sequences} of length $n$ is denoted by $\maxweightset{n}$. 
\end{definition}

For $1 \leq n \leq 4$, the set $\maxweightset{n}$ of maximum weight jump-float sequences is
\begin{align*}
\maxweightset{1} &= \{ (1) \} \\
\maxweightset{2} &= \{ (0,2), (1,1)   \} \\
\maxweightset{3} &= \{   (0,2,2) , (1,0,3) \} \\
\maxweightset{4} &= \{    (1,0,3,3)  \} 
%\maxweightset{5} &= \{  (0,2,0,4,4),  (1,0,3,3,3),  (1,1,0,4,4) \},
\end{align*}
and these are the jump-float sequences that correspond to $\pattern$-avoiding permutations with maximum inversion number. Having established this equivalence, we turn our attention to characterizing $\maxweightset{n}$ for the remainder of this work.

\subsection{Rephrasing Theorem~\ref{thrm:correct-perm}}

We now set to rephrasing Theorem~\ref{thrm:correct-perm} in terms of $w^*_n$ and the size of set $\maxweightset{n}$. Lemma~\ref{lem:jump-float-bijection} shows that $w_n^* = \maxinv(\permset{n+1}(\pattern))$ and $|\maxweightset{n}| = |\mathcal{Z}_{n+1}(\pattern)|$. We begin by illuminating the origin of the parameters $\row{n}$ and $\ind{n}$ of Theorem~\ref{thrm:correct-perm}: they describe the location of the jump-float sequence of size $n$ in the following trapezoid.

\begin{definition}
\label{def:size-trapezoid}
Arrange the sequence $\{1,2,3, \ldots\}$ of positive integers into the \emph{size trapezoid} $S(k,\ell)$ where $1 \leq k$ and $0 \leq \ell \leq k+1$, so that row $k$ has $k+2$ entries. The first four rows of $S(k,\ell)$ are
\begin{center}
\begin{tikzpicture}[scale=.8]
%\node at (-4, -1.125) {$S=$};

\node at (0,0) (1){$1$};
\node at (1.5,0) (2){$2$};
\node at (3,0) (3){$3$};
%\node at (4.5,0) (4a){$1$};
\node at (-0.75, -0.75) (4b){$ 4$};
\node at (0.75, -0.75) (5){$ 5$};
\node at (2.25, -0.75) (6){$6$};
\node at (3.75, -0.75) (7){$7$};
%\node at (5.25, -0.75) (8a){$1$};
\node at (-1.5, -1.5) (8b){$8$};
\node at (0, -1.5) (9){$9$};
\node at (1.5, -1.5) (10){$10$};
\node at (3, -1.5) (11){$11$};
\node at (4.5, -1.5) (12){$12$};
%\node at (6, -1.5) (13a){$1$};
\node at (-2.25, -2.25) (13b){$13$};
\node at (-0.75, -2.25) (14){$14$};
\node at (0.75, -2.25) (15){$15$};
\node at (2.25, -2.25) (16){$16$};
\node at (3.75, -2.25) (17){$17$};
\node at (5.25, -2.25) (18){$18$.};
%\node at (6.75, -2.25) (19a){$1^$};
\end{tikzpicture}
\end{center}
The initial entry in row $k$ is $S(k,0) =  \T_{k+1}-2$ where $
\T_j = { j+1 \choose 2} = \sum_{i=1}^j i$
is the $j$th triangular number. 
The final entry in row $k$ is $S(k,k+1) = \T_{k+1} - 2 + (k+1) = \T_{k+2} -3$.
\end{definition}

 This trapezoidal organization is justified by the following observation. By placing $|\mathcal{Z}_{n+1}(\pattern)| = |\maxweightset{n}|$ from equation~\eqref{eqn:count-seq} in the location of $n$ in $S(k,\ell)$, a Pascal-like recurrence appears:
\begin{equation}
\label{eqn:K-trapezoid}
\begin{tikzpicture}[scale=.8]

%\node at (-4, -1.125) {$K=$};

\node at (0,0) (1){$1$};
\node at (1.5,0) (2){$2$};
\node at (3,0) (3){$2$};
%\node at (4.5,0) (4a){$1$};
\node at (-0.75, -0.75) (4b){$ 1$};
\node at (0.75, -0.75) (5){$ 3$};
\node at (2.25, -0.75) (6){$4$};
\node at (3.75, -0.75) (7){$3$};
%\node at (5.25, -0.75) (8a){$1$};
\node at (-1.5, -1.5) (8b){$1$};
\node at (0, -1.5) (9){$4$};
\node at (1.5, -1.5) (10){$7$};
\node at (3, -1.5) (11){$7$};
\node at (4.5, -1.5) (12){$4$};
%\node at (6, -1.5) (13a){$1$};
\node at (-2.25, -2.25) (13b){$1$};
\node at (-0.75, -2.25) (14){$5$};
\node at (0.75, -2.25) (15){$11$};
\node at (2.25, -2.25) (16){$14$};
\node at (3.75, -2.25) (17){$11$};
\node at (5.25, -2.25) (18){$5.$};
%\node at (6.75, -2.25) (19a){$1^$};
\end{tikzpicture}
\end{equation}
Indeed, the proof of Theorem~\ref{thrm:correct-perm} hinges on establishing this Pascal-like recurrence for $|\maxweightset{n}|$. Our first order of business is to provide formulas for the row and column of size $n$ in trapezoid $S$.

%It will be beneficial to accurately place the location of $n$ in $S$. 
%We provide formulas for the row $\row{n}$ and column $\ind{n}$ of $n$ in size trapezoid $S$ in the following lemma.

\begin{lemma}
\label{lem:row-ind}
For a positive integer $n$, its row $\row{n}$ and column $\ind{n}$ in size trapezoid $S$ are given by
$$
\row{n} = \left\lfloor \frac{-3 + \sqrt{8n+17}}{2} \right\rfloor
\qquad
\mbox{and}
\quad
\ind{n} = n + 2 - \T_{\row{n}+1}.
$$
where the row index starts at 1 and the column index starts at 0.
\end{lemma}

\begin{proof}
For $k \geq 1$, the sequence $ \{ S(k,0) \} =  \{ 1,4,8,13,\ldots \}$ of first entries is $\{ \T_{k+1}-2 \}$.
Solving for the case $n = \T_{k+1} -2$ for $k$ yields
$$
k = \frac{-3 + \sqrt{8n+17}}{2}.
$$
For general $n$, let $k$ be the unique integer such that
$\T_{k+1}-2 \leq n \leq \T_{k+2}-3$. We have $\row{n}=k$ and
$0 \leq \ind{n} \leq k+1$, which is the row and column indexing that we desire.
\end{proof}

With these parameters defined, we turn to restating our main result, using the language of jump-float sequences.  First, we resolve the sizes $n$ for which there is a unique maximum weight sequence. These are special and deserving of a definition.

\begin{definition}
\label{def:perfect}
The \emph{perfect sequence} for $n=\T_{k+1}-2$ is the maximal sequence $p \in \maxset{n}$   given by
$$
p =(1,0,\underbrace{3,3}_2,0,\underbrace{6,6,6}_3,0, \underbrace{10,10,10,10}_4,0, \ldots, 0, \underbrace{\T_k, \ldots, \T_k}_k)
$$
with weight
$$
w(p)=\sum_{i=1}^k i\T_i =\sum_{i=1}^k i \left(\frac{i(i+1)}{2}\right) = \frac{k(k+1)(k+2)(3k+1)}{24}.
$$
\end{definition}

Note that for $n=\T_{k+1}-2$, we have $\row{n}=k$ and $\ind{n}=0$. So perfect sequences correspond to the first entry of each row of the trapezoid.  We also note that the sequence of weights $\{ w(p) \}$ of perfect sequences is OEIS A001296 \cite{oeis}.

\begin{theorem}
    \label{thrm:unique}
Let $n=\T_{k+1}-2$ for some $k\geq 1$. Then $|\maxweightset{n}|=1$, and the unique maximum weight jump-float sequence is the perfect sequence for $n$
with weight $w_n^* = \sum_{i=1}^k i\T_i.$

\end{theorem}

The general case is covered by the following theorem.

\begin{theorem}
    \label{thrm:correct}
    For $n \geq 1$, the number of maximum weight jump-float sequences is
    \[
    |\maxweightset{n}|=\binom{\row{n}}{\ind{n}}+\binom{\row{n}+1}{\row{n}+2-\ind{n}}
    =
    \binom{\row{n}}{\ind{n}}+\binom{\row{n}+1}{\ind{n}-1}
    \]
    with weight
    \[
    \maxweight{n}=\T_{n-\row{n}} + \sum_{i=1}^{\row{n}}  \T_i
    = \sum_{j=1}^{n-\row{n}} j + \sum_{i=1}^{\row{n}}  \T_i,
    \]
    where $\T_i = {i+1 \choose 2}$ is the $i$th triangular number, and  $\rho_n$ and $\lambda_n$ are defined in Lemma~\ref{lem:row-ind}.
\end{theorem}

Note that Theorem~\ref{thrm:correct-perm} is equivalent to Theorem~\ref{thrm:correct}, thanks to Lemma~\ref{lem:jump-float-bijection}. As such, the remainder of the paper is devoted to proving Theorems~\ref{thrm:unique} and~\ref{thrm:correct}.

%% file: max-weight.tex
\section{Maximum Weight Jump-Float Sequences}
\label{sec:max-weight}

In Section~\ref{sec:jump-float}, we pivoted our attention from $\pattern$-avoiding permutations to jump-float sequences, which more obviously showcase the inversion number of the permutation. Our goal for this section is to characterize the highly structured forms of the maximum weight jump-float sequences. 

To this end, we define a simple, equivalent formulation called a \emph{block sequence} in Section~\ref{sec:block-seq}, which encodes the lengths of nonzero subsequences of a jump-float sequence. In Section~\ref{sec:max-block-seq}, we prove a strong restriction on the values in the block sequence, showing that the entries cannot differ too much from their indices in the block sequence. 

Finally, in Section~\ref{sec:max-block-row}, we use these conditions on the block sequence to prove restrictions on the structure of the jump-float sequence. Specifically, if $n$ is in row $\row{n}$ of size trapezoid $S$, both the length of and the last entry of the block sequence is either $\row{n}$ or $\row{n}+1$.

% In Section~\ref{sec:jump-float}, we pivoted our attention from $\pattern$-avoiding permutations to maximal jump-float sequences $\maxset{n}$ and maximum weight jump-float sequences $\maxweightset{n}$. Before we can prove our main results, we must further understand the nature sequences in $\maxset{n}$ and $\maxweightset{n}$.
% In Section~\ref{sec:block-seq}, we take advantage of the simple structure of maximal jump-float sequences and define an equivalent formulation called block sequences. 
% In Section~\ref{sec:max-block-seq}, we prove a strong restriction on the entries of maximum weight block sequences.
% In Section~\ref{sec:max-block-row}, we show that the row $\row{n}$ of size $n$ in trapezoid $S$ of Definition~\ref{def:size-trapezoid} narrows down the length of a maximum weight block sequence, as well as the size of its final entry.

\subsection{Block sequences}
\label{sec:block-seq}

Lemma~\ref{lemma:max-skip-sequence} shows that  maximal sequences in $\maxset{n}$ have a very specific structure. Each maximal sequence $s=(s_1, \ldots, s_n)$
can be decomposed into constant subsequences, separated by single zero entries. The first entry $s_1$ can be zero, but the last entry $s_n$ cannot. Also, the value of a constant subsequence is equal to the index where the subsequence begins. Given this constrained structure, we can  efficiently represent a maximal sequence by listing the lengths of its constant subsequences.

\begin{figure}[h!]
    \centering
\begin{tikzpicture}[scale=.6]
        \tri[(0, 0)]{0,2,0,0,0}{violet!50};
        \tri[(0, 0)]{0,0,0,4,4}{teal!50};
        \node at (-.2,-.7) {\small $ b=(\textcolor{violet}{1},\textcolor{teal}{2})$};

        \tri[(3, 0)]{0,2,2,0,0}{violet!50};
        \tri[(3, 0)]{0,0,0,0,5}{teal!50};
        \node at (2.8,-.7) {\small $b=(\textcolor{violet}{2},\textcolor{teal}{1})$};

        \tri[(6, 0)]{0,2,2,2,2}{violet!50};
        \node at (5.9,-.7) {\small $b=(\textcolor{violet}{4})$};
        
        \tri[(9, 0)]{1,0,0,0,0}{violet!50};
        \tri[(9, 0)]{0,0,3,0,0}{teal!50};
        \tri[(9, 0)]{0,0,0,0,5}{blue!50};
        \node at (8.9,-.7) {\small $b=(\textcolor{violet}{1},\textcolor{teal}{1},\textcolor{blue}{1})$};
        
        \tri[(12, 0)]{1,0,0,0,0}{violet!50};
        \tri[(12, 0)]{0,0,3,3,3}{teal!50};
        \node at (11.9,-.7) {\small $b=(\textcolor{violet}{1},\textcolor{teal}{3})$};

        \tri[(15, 0)]{1,1,0,0,0}{violet!50};
        \tri[(15, 0)]{0,0,0,4,4}{teal!50};
        \node at (14.9,-.7) {\small $b=(\textcolor{violet}{2},\textcolor{teal}{2})$};

        \tri[(18, 0)]{1,1,1,0,0}{violet!50};
        \tri[(18, 0)]{0,0,0,0,5}{teal!50};
        \node at (17.9,-.7) {\small $b=(\textcolor{violet}{3},\textcolor{teal}{1})$};
        
        \tri[(21, 0)]{1,1,1,1,1}{violet!50};
        \node at (20.9,-.7) {\small $b=(\textcolor{violet}{5})$};
\end{tikzpicture}
    \caption{The triangle representations of the maximal sequences from $\maxset{5}$ and their corresponding block sequences from $\blockset{5}$. }
\label{fig:maxset-5}
\end{figure}

\begin{definition}
\label{def:block-seq}
A \emph{block} of $s \in \maxset{n}$ is a maximal nonzero subsequence of $s$. The length of this subsequence is the \emph{width} of the block, while the nonzero value of the subsequence is the \emph{height} of the block. Distinct blocks are separated by a single zero, called a \emph{gap}. If $s_1=0$, then it is called a \emph{leading gap}.
Given $s \in \maxset{n}$, its  \emph{block sequence} $b=b(s)=(b_1,\ldots b_k)$ is the list of widths of the blocks of $s$. We write $|b|=k$ and for convenience, we define its weight $w(b)=w(s)$. The set of \emph{maximal block sequences} is
$$
\blockset{n} = \{ b(s) : s \in \maxset{n} \}
$$
and the set of \emph{maximum weight block sequences} is
$$
\mwblockset{n} = \{ b(s) : s \in \maxweightset{n} \}.
$$
\end{definition}

Figure~\ref{fig:maxset-5} shows the block sequences $\blockset{5}$ for sequences of $\maxset{5}$, as well as their triangular representations, where each block is a rectangle of filled squares. We have
$\mwblockset{5} = \{ (1,3), (2,2) \}$, and these block sequences have weight 10.

Looking at block sequences will allow us to say even more about the structure of maximum weight sequences. Therefore, we shift our focus to $\blockset{n}$ and $\mwblockset{n}$, rather than $\maxset{n}$ and $\maxweightset{n}$. Our next two lemmas show that we can recover a maximal sequence $s \in \maxset{n}$ from its block sequence $b \in \blockset{n}$.

%  For example, the  maximal sequences
% $$
% (1,1,0,4,0,6,6,6,0,10,10,10) \quad \mbox{and} \quad
% (0,2,2,2,0,6,6,0,9,9,0,12)
% $$
% from $\maxset{12}$ have block sequences 
% $$
% (2,1,3,3) \quad \mbox{and} \quad (3, 2, 2, 1).
% $$

\begin{lemma}
\label{lem:alg}
    If $b=(b_1, b_2, \ldots, b_k) \in \blockset{n}$ then $n = \gamma + \sum_{i=1}^k b_i$ where $\gamma \in \{ k-1, k \}$ is the number of gaps in the corresponding maximal sequence.
\end{lemma}

\begin{proof}
Let $s \in \maxset{n}$ be the maximal sequence for block sequence $b \in \maxset{n}$. It is clear that $s$ cannot have a subsequence of multiple zeros: otherwise we could increase the first of these entries to 1 to create a valid jump-float sequence $s' \succ s$.

The difference $n - \sum_{i=1}^k b_i$ is the number of zero entries $s$. There are $k-1$ gaps between the blocks. It is also possible to have an additional leading gap before the first block. 
%So the total number of gaps is either $k-1$ or $k$.
\end{proof}

\begin{lemma}
\label{lem:height-seq}
Given $b = (b_1, b_2, \ldots, b_k) \in \blockset{n}$, we can determine its corresponding $s \in \maxset{n}$. In particular, if $h_i$ is the height of block $i$, then
$h_1 \in \{1,2\}$ and $h_i = h_{i-1} + b_i + 1$ for $2 \leq i \leq k$.
\end{lemma}

\begin{proof}
Let $b = (b_1, b_2, \ldots, b_k) \in \blockset{n}$, and set $\gamma = n - \sum_{i=1}^k b_i$.  Observe that the height $h_i$ of the $i$th block \emph{also} equals the index where the block starts.

We recursively determine the heights $h_1, h_2, \ldots, h_k$ of the blocks.
When $\gamma=k-1$, the first block starts at index $h_1=1$. Otherwise, $\gamma=k$ and the first block starts at index $h_1=2$. For $2 \leq i \leq k$, block $i$ starts at $h_i = h_{i-1} + b_{i-1} + 1$. 
This completely determines all of the values of the corresponding $s \in \maxset{n}$.
\end{proof}

For example, consider block sequence $b = (4,1,2) \in \blockset{10}$. We have $k=3$ and $\gamma=10-7=3,$ so there is a leading gap. Then $h_1=2, h_2=7$, and $h_3=9$. The corresponding maximal jump-float sequence is $s=(0,2,2,2,2,0,7,0,9,9)$, and $w(b)=w(s)=8+7+18=33$.

As another example, recall that the perfect sequences (Definition~\ref{def:perfect}) are a particularly important family of maximal sequences. For $n=\T_{k+1}-2$, the perfect sequence
$$
p =(1,0,\underbrace{3,3}_2,0,\underbrace{6,6,6}_3,0, \underbrace{10,10,10,10}_4,0, \ldots, 0, \underbrace{\T_k, \ldots, \T_k}_k).
$$
has block sequence $b=(1,2,3,\ldots,k)$ and has $k-1$ gaps. 
The first three perfect sequences are shown in Figure~\ref{fig:unique_ex}. 

\begin{figure}[h!]
    \centering
\begin{tikzpicture}[scale=.6]
        \tri[(-6, -0.5)]{1}{violet!50};
        \node (eq) at (-6.5,-2) {\setlength{\arraycolsep}{1pt} \renewcommand{\arraystretch}{1.2}$\begin{array}{rcl} s &=&(\textcolor{violet}{1}) \\ b(s)&=&(\textcolor{violet}{1}) \\ \phi(s)&=&21  \end{array}$};

        \tri[(0,-0.5)]{1,0,0,0}{violet!50};
        \tri[(0,-0.5)]{0,0,3,3}{teal!50};
        \node (eq) at (0,-2) {\setlength{\arraycolsep}{1pt} \renewcommand{\arraystretch}{1.2}$\begin{array}{rcl} s &=&(\textcolor{violet}{1},0,\textcolor{teal}{3,3}) \\ b(s)&=&(\textcolor{violet}{1},\textcolor{teal}{2}) \\ \phi(s)&=&45132  \end{array}$};

        \tri[(8,-0.5)]{1,0,0,0,0,0,0,0}{violet!50};
        \tri[(8,-0.5)]{0,0,3,3,0,0,0,0}{teal!50};
        \tri[(8,-0.5)]{0,0,0,0,0,6,6,6}{blue!50};
        \node (eq) at (8,-2) {\setlength{\arraycolsep}{1pt} \renewcommand{\arraystretch}{1.2} $\begin{array}{rcl} s &=&(\textcolor{violet}{1},0,\textcolor{teal}{3,3},0,\textcolor{blue}{6,6,6}) \\ b(s)&=&(\textcolor{violet}{1},\textcolor{teal}{2},\textcolor{blue}{3}) \\ \phi(s)&=&789156243  \end{array}$};

        % \tri[(15,-0.5)]{1,0,3,3,0,6,6,6,0,10,10,10,10}{violet!50};
        % \node (s3) at (15,-1) {$s=(1,0,3,3,0,6,6,6,0,10,10,10,10)$};
        % \node (b3) at (15,-1.75) {$b=(1,2,3,4)$};        
  
\end{tikzpicture}
    \caption{The perfect sequences for $n=1,4,8$, each shown with its triangle representation, block sequence, and permutation.}
    \label{fig:unique_ex}
\end{figure}
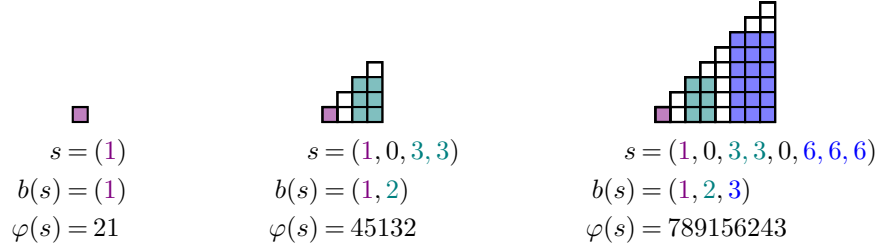

\subsection{Restricting the entries of maximum weight block sequences}
\label{sec:max-block-seq}

The block sequence provides a convenient way to describe the lengths of nonzero subsequences of a maximum weight jump-float sequence. It turns out that the requirement of maximum weight leads to the entries of the block sequence being highly restricted. The goal of this subsection is to prove the following proposition.

\begin{prop}
    \label{prop:block_index}
    If $b=(b_1,\ldots b_k) \in \mwblockset{n}$ is a maximum weight block sequence then $b_i=i+e_i$, where $e_i \in \{-1,0,1\}$. 
\end{prop}

In words, the entry of the block sequence cannot be all that different from its index in the sequence. The proof of this proposition will require several structural results. First, we prove a general lemma about how the weight changes when the sizes of successive blocks of $b \in \blockset{n}$ are changed.

%The lemma considers the net change in weight of a maximal sequence when its block sequence $b=(b_1, b_2, \ldots, b_k)$ is changed slightly.

\begin{lemma}
\label{lem:block-delta}
Let  $b=(b_1, \ldots, b_k) \in \blockset{n}$. Let $b'=(b_1', \ldots, b_k') \in \blockset{n}$ where 
\[b_j' = \left\{ 
\begin{array}{ll}
b_{i} +\delta & \mbox{for } j = i, \\
b_{i+1} -\delta & \mbox{for } j = i+1, \\
b_j & \mbox{otherwise,}
\end{array}
\right.\]
for $1 \leq \delta < b_{i+1}$. Then the change in weight is
\[
w(b') - w(b) = \delta( b_{i+1} - b_{i} -1 - \delta).
\]
\end{lemma}

\begin{proof}
Let $(h_1, \ldots, h_k)$ and $(h_1', \ldots h_k')$ be the height sequences for $b$ and $b'$, respectively.
By Lemma~\ref{lem:height-seq}, we have
\[
h_j' = h_j \quad \mbox{for } j \neq i+1
\]
and 
\[
h_{i+1}' = 1+b'_{i} +h_{i}' = 1+ (b_{i} +\delta)  + h_{i} = h_{i+1} + \delta.
\]
Therefore
\begin{align*}
w(b') - w(b) &=
\sum_{j=1}^k b'_j h'_j - \sum_{j=1}^k b_j h_j 
    = b'_{i}h'_{i} + b'_{i+1}h'_{i+1}  - \big( b_{i}h_{i} + b_{i+1}h_{i+1} \big) \\
    & = (b_{i}+\delta) h_{i} - b_{i} h_{i}  +  (b_{i+1}-\delta)(h_{i+1}+\delta)  - b_{i+1}h_{i+1} \\
    &= \delta ( h_{i} -  h_{i+1} +  b_{i+1} - \delta) 
    = \delta( h_{i} - (1 +b_{i} + h_{i}) +b_{i+1} - \delta) \\
    &= \delta( b_{i+1} - b_{i} -1 - \delta),
\end{align*}
which completes the proof.
\end{proof}

Next, we have three corollaries about the variation between successive blocks of maximum weight sequences. First, we prove that the block sequence is weakly increasing, but by no more than two at every entry. 

\begin{cor}
\label{cor:block_var}
If $b = (b_1,\ldots, b_k) \in \mwblockset{n}$  then
\[ 
0 \leq b_{i+1}-b_{i} \leq 2
\]
for  $1 \leq i \leq k-1$. 
\end{cor}

\begin{proof}
Since $b$ has maximum weight, we have $0 \geq w(b') - w(b)$ for every $b' \in \mwblockset{n}$. In particular, for the block sequence $b'$  in Lemma~\ref{lem:block-delta} with  $\delta=1$, we have
$$
0 \geq w(b') - w(b) = b_{i+1} - b_{i} -2
$$
which is equivalent to $b_{i+1} - b_{i} \leq 2$. On the other hand, taking $\delta=-1$ in Lemma~\ref{lem:block-delta} gives
$$
0 \geq w(b') - w(b) = -(b_{i+1} - b_{i})
$$
which is equivalent to $ b_{i+1} - b_i \geq 0$. 
% Putting these inequalities together tells us that
% $$
% 0 \leq b_{i+1} - b_{i} \leq 2
% $$
% for the maximal block sequence $s$.
\end{proof}

%Our next lemma states that when a maximal jump-float sequence has a block sequence with consecutive equal blocks, there exists a another maximal jump-float sequence where these two blocks are not equal. 

The subsequent two corollaries describe the behavior of successive equal blocks in a maximum weight block sequence.

\begin{cor}
    \label{cor:block_shift}
    Suppose that $b=(b_1,\ldots b_k)  \in \mwblockset{n}$ such that $b_{i}=b_{i+1}$ for some $i$. Then  $b' \in \blockset{n}$ with block sequence
    \[
    b_j'=\left\{ \begin{array}{ll} b_{i}-1 & \text{ for } j=i, \\ b_{i+1}+1 & \text{ for } j=i+1, \\ b_j & \text{ otherwise } \end{array} \right.
    \]
    has weight $w(b') = w(b),$ and hence $b' \in \mwblockset{n}$ as well.
\end{cor}

\begin{proof}  By Lemma~\ref{lem:block-delta} with $\delta=-1$, $w(b')=w(b)-(b_{i+1}-b_{i})=w(b)$. 
\end{proof}

%Further, we claim that no jump-float sequence with three consecutive equal blocks can be a maximum weight jump-float sequence. 

\begin{cor}
\label{cor:no_three_equal}
If $b=(b_1,\ldots,b_k)  \in \blockset{n}$ with $b_{i}=b_{i+1}=b_{i+2}$, then $b \notin \mwblockset{n}$.
\end{cor}

\begin{proof}
Let $b' \in \blockset{n}$ be the block sequence constructed in Corollary~\ref{cor:block_shift}, which has $w(b')=w(b)$.
Observe block sequence $b'$ has $b_{i+1} > b_{i+2}$. By Corollary~\ref{cor:block_var}, block sequence $b'$ does not have maximum weight, so neither does $b$.
%Consider the jump-float sequence $s'$ with block sequence 
%\[
%b_j'=\left\{ \begin{array}{l l} b_{i-1}-1 & \text{ if } j=i-1,  \\ b_i+1 & \text{ if } j=i, \\ b_i & %\text{ otherwise. }\end{array} \right.
%\]  
%By Corollary~\ref{cor:block_shift}, $w(s)=w(s')$. Observe that for jump-float sequence $s'$, we have $b_{i}' > b_{i+1}'$. By Corollary~\ref{cor:block_var}, the jump-float sequence $s'$ does not have maximum weight, so neither does $s$.
\end{proof}

We use these corollaries to further classify when the block increases by zero and when it increases by two. Specifically, these behaviors alternate within the block sequence.

%We are now ready to prove the two main results of this section. First, the block sequence of a maximum weight jump-float sequence must alternate increasing by two and staying constant.

\begin{prop}
    \label{prop:alternate}
    Let $b=(b_1,\ldots ,b_k) \in \mwblockset{n}$. 
    \begin{enumerate}[(a)]
    \item Suppose that there exists $i<j$ such that $b_i=b_{i-1}$ and $b_{j+1}=b_{j}$. Then there exists an $i\leq r<j$ such that $b_{r+1}=b_r+2$.
    \item Suppose that there exists $i<j$ such that $b_i=b_{i-1}+2$ and $b_{j+1}=b_{j}+2$. Then there exists an $i\leq r<j$ such that $b_{r+1}=b_r$. 
    \end{enumerate}
\end{prop}

\begin{proof}
The proofs of both statements use induction and proof by contradiction. 
Let $b=(b_1,\ldots ,b_k) \in \mwblockset{n}$.

We begin with (a). Suppose for the sake of contradiction that for all $r$ such that $i\leq r <j$, we have $b_{r+1} = b_r+1$. We proceed with induction on $d=j-i$. For $d=1$, define a new block sequence $b' \in \blockset{n}$ as
\[
b'_\ell=\left\{\begin{array}{ll} b_{i-1}-1 & \text{ for } \ell=i-1, \\ b_{i}+1 & \text{ for } \ell=i, \\ b_\ell & \text{ otherwise. }\end{array}\right.
\] 
By Corollary~\ref{cor:block_shift}, $w(b)=w(b')$, but $b'_{i}=b'_j=b'_{j+1}$. By Corollary~\ref{cor:no_three_equal}, neither $b'$ nor $b$ is maximum weight. Now consider $d=m+1.$ We define the same modified block sequence $b'$  as above, and again note that by Corollary~\ref{cor:block_shift}, $w(b)=w(b')$. Note, however, that $b'_{i}=b'_{i+1}$, which means that block sequence $b'$ falls into the case where $d=m$. By induction, $b'$ does not have maximum weight, and neither does $b$. This is a contradiction, and as such, the claim holds. 

For (b),  suppose for the sake of contradiction, that for all $r$ such that $i\leq r <j$, $b_{r+1} = b_r+1$. We proceed by induction on $d=j-i$. For $d=1$, define a new block sequence $b' \in \blockset{n}$  as 
\[
b'_\ell=\left\{\begin{array}{ll} b_{i}+1 & \text{ for } \ell=i, \\ b_{i+1}-1 & \text{ for } \ell=i+1, \\ b_\ell & \text{ otherwise. } \end{array}\right.
\]
By Corollary~\ref{cor:block_shift}, $w(b)=w(b')$, but $b'_{j+1}=b'_j+3$. By Corollary~\ref{cor:block_var}, neither $b'$ nor $b$ are maximum weight. Now consider $d=m+1.$ We define the same modified block sequence $b'$  as above, and again note that by Corollary~\ref{cor:block_shift}, $w(b)=w(b')$. Note, however, that $b'_{i+1}=b'_{i}+2$, which means the sequence $b'$ falls into the case where $d=m$. By induction, $b'$ does not have maximum weight, and neither does $b$. This is a contradiction, and as such, the claim holds. 
\end{proof}

Proposition~\ref{prop:alternate} provides a strong restriction on how the entries of a maximum weight block sequence vary from one to the next. We can now prove the main result of this subsection, Proposition~\ref{prop:block_index}, that block widths of $b \in \mwblockset{n}$ differ from their index by at most 1.

\begin{proof}[Proof of Proposition~\ref{prop:block_index}]
   Let $i$ be the smallest index such that $b_i\neq i+e_i$ for $e_i \in \{-1,0,1\}.$

   {\bf Case 1:} Suppose $b_i<i-1$. By the definition of $i$, we know $b_{i-1}\geq (i-1)-1$. Since $b$ is weakly increasing by Corollary~\ref{cor:block_var}, it follows that $b_i=b_{i-1}=i-2$. We now consider how the block sequence varies term to term. Let $\alpha_j=b_{j+1}-b_j$. Note, $\alpha_{i-1}=0$. By Corollary~\ref{cor:block_var}, we know $\alpha_j \in \{0,1,2\}$, so $\{\alpha_j : j \in [i-2]\}$ is a multiset of 0, 1, and 2. Define $x_0$ to be the cardinality of 0 in this multiset, and similarly define $x_1$ and $x_2$. Immediately, we know $x_0+x_1+x_2=i-2$. Further, we also know $x_1+2x_2=i-2-b_1\leq i-3$. Algebraic manipulation yields that $x_0\geq x_2 +1$. However, by Proposition~\ref{prop:alternate}, we know $|x_0-x_2|\leq 1$. As such, $x_0=x_2+1$, and there is one more jump by 0 than there is by 2. However, this is a contradiction of Proposition~\ref{prop:alternate} (a), since there will be two more zeros than twos in the multiset $\{\alpha_j : j \in [i-1]\}$. 
   %\todo{there is probably a better way to invoke pigeonhole here, but it's definitely works this way too.}

   {\bf Case 2:} Suppose $b_i>i+1$. Then, again by the definition of $i$, we know $b_{i-1}\leq (i-1)+1$. By Corollary~\ref{cor:block_var}, it follows that $b_{i-1}=i$ and $b_{i}=i+2$. A similar argument to Case 1 can be used to contradict Proposition~\ref{prop:alternate} (b), forcing two more twos than zeros in the multiset $\{\alpha_j : j \in [i-1]\}$. 
   %\todo{this might be too much sweeping under the rug, but it's really gonna be the same thing.}
\end{proof}

%%%%%%%
%%%%%%%
%%%%%%%
\subsection{Maximum weight block sequences in a given row}
\label{sec:max-block-row}

With a more thorough understanding of the structure of the block sequence, we return to the jump-float sequence. If a block sequence $b$ has length $k$, and each entry has the property that $b_i=i+e_i$ for $e_i \in \{-1,0,1\}$, it stands to reason that we are able to bound the length of the the jump-float sequence associated with $b$. 

In this section, we  classify which sizes $n$ of  $b \in \mwblockset{n}$ can have $k$ blocks. We will see that $k$  is related to the row of $n$ in the size trapezoid $S$ from Definition~\ref{def:size-trapezoid}. Recall that the values $(\row{n}, \ind{n})$ from Lemma~\ref{lem:row-ind} are the row and column of size $n$ in trapezoid $S$.

\begin{prop}
\label{prop:k_blocks_length_n}
If  $b \in \mwblockset{n}$ with length $|b|=k$, then
\begin{equation}
\label{eqn:k_blocks_length_n}
\T_k-1\leq n\leq \T_{k+2}-3.  
\end{equation}
\end{prop}

In other words, if a maximum weight sequence $s \in \maxweightset{n}$ has $k$ blocks, then $\row{n} \in \{ k-1, k \}$, so that $n$ is in either row $k-1$ or row $k$ of size trapezoid $S$ from Definition~\ref{def:size-trapezoid}.

\begin{proof}
Let $b= (b_1, \ldots, b_k)$. 
By Proposition~\ref{prop:block_index}, we have
$$
\sum_{i=1}^k b_i = \sum_{i=1}^k (i + e_i)  = \T_k  + \sum_{i=1}^k e_i
$$
where $-k\leq \sum_{i=1}^ke_i\leq k$. By Lemma~\ref{lem:alg},
\[ 
n-k \leq \T_k  + \sum_{i=1}^k e_i \leq n-k+1.
\]
The left inequality implies that $n-k  \leq  \T_k + k$ which is equivalent to $n  \leq \T_k + 2k = \T_{k+2} -3$.
The right inequality implies that
$\T_k - k  \leq n-k+1$
which is equivalent to $\T_k - 1 \leq n$. 
%To finish the proof of the claim, recall that by Definition~\ref{def:size-trapezoid}, row $j$ of size trapezoid $S$ starts with $n=\T_{j+1}-2$ and ends with \blue{$n=\T_{j+2}-3$}. Thus, by the proven inequality, $n$ lies in either row $k-1$ or row $k$.
\end{proof}

The next two corollaries characterize the number of blocks for maximum weight sequence whose size $\T_{j+1} -2 \leq n \leq \T_{j+2}-3$ is in row $j$ of size trapezoid $S$. As per Lemma~\ref{lem:row-ind}, this condition is equivalent to $\row{n} = \T_{j+1}-2$ and
$0 \leq \ind{n} \leq \rho_{n}+1.$

\begin{cor}
\label{cor:perfect_blocks}
    If $n=\T_{j+1}-2$ then $b \in \mwblockset{n}$ has length $|b|=j$.
\end{cor}

\begin{proof}
For $n=\T_{j+1}-2$, equation~\eqref{eqn:k_blocks_length_n} only holds when $k=j$.
\end{proof}

\begin{cor}
\label{cor:row_to_blocks}
    If $\T_{j+1}-1 \leq n \leq \T_{j+2}-3$ then $b \in \mwblockset{n}$ has length $|b| \in \{ j, j+1 \}$. 
\end{cor}

\begin{proof}
For $\T_{j+1}-1 \leq n \leq \T_{j+2}-3$, equation~\eqref{eqn:k_blocks_length_n} holds only when $k=j$ or $k=j+1$.
\end{proof}

Next, we have the following result, complementary to  Proposition~\ref{prop:k_blocks_length_n}.

\begin{prop}
    \label{prop:large_block}
    The last entry of $b \in \mwblockset{n}$ is either $\row{n}$ or $\row{n}+1$.
\end{prop}

\begin{proof}
Consider $s \in \maxweightset{n}$, with block sequence $b \in \mwblockset{n}$. Let $\row{n} = k$, so that $\T_{k+1}-2\leq n \leq \T_{k+2}-3$. By Corollary~\ref{cor:row_to_blocks}, we know $|b| \in \{k,k+1\}$. Combining this with Proposition~\ref{prop:block_index}, we consider six possible cases: if $|b|=k$, then the last block $b_k$ is either $k-1$, $k$, or $k+1$, or if $|b|=k+1$, then the last block $b_{k+1}$ is either $k$, $k+1$, or $k+2$. Only two of these cases do not fall under the claim, and we will prove neither can occur.

\textbf{Case 1:} Suppose $|b|=k$ and $b_k=k-1$. We claim that $b_i\leq i$ for all $i$, and we proceed by contradiction. Suppose $j$ is the largest index such that $b_j >j$. By Proposition~\ref{prop:block_index}, it follows that $b_j=j+1$. Further, by the definition of $j$ and by Corollary~\ref{cor:block_var} and Proposition~\ref{prop:block_index}, $b_{j+1}=j+1$ and for every index $\ell >j$, $b_\ell \in \{\ell-1, \ell\}$. Consider the first index where this value is $\ell-1$. We know such an index exists, because $b_{k}=k-1$. For convenience of indexing, we consider this location to be $b_{m+1}=m.$ Note, then, that $b_p=p$ for all $j+1\leq p \leq m$. However, this contradicts Proposition~\ref{prop:alternate} part (a). Thus $b_i \leq i$ for all indices $i$. 

With this bound on the sizes of the blocks, we consider a bound for $n$. By Lemma~\ref{lem:alg}, we know 
\[n=\beta+\sum_{i=1}^k b_i \leq \beta +\sum_{i=1}^{k-1} i +k-1 =\beta+\T_{k}-1\]
where $\beta \in \{k-1,k\}$. We note that if $\beta=k-1$, or if $b_i<i$ for any $i$, we have an immediate contradiction to $\row{n}=k$. However, we must evaluate the case where $\beta=k$ and $b_i=i$ for all $i\leq k-1$. In this case, $n=k+\T_k-1=\T_{k+1}-2$, and $b=(1,2,\ldots,k-1,k-1)$, and $s_1=0$. 

We claim this sequence cannot be a maximum weight sequence: it is the same length as the perfect sequence for $n$ (with block sequence $(1,2,\ldots, k-1, k)$), yet has a smaller weight by exactly one. Note that with these conditions, the block sequence is uniquely determined, $b=(1,2,\ldots, k-1,k-1)$, with $k$ gaps, meaning $s_1=0$. We consider how transforming this sequence to the perfect sequence (see Definition~\ref{def:perfect}) with block sequence $(1,2,\ldots, k-1, k)$ changes the weight, but use the triangle interpretation of the sequence to clarify the counting. By removing the bottom row of the triangle, and adding on a column that increases the last block from width $k-1$ to $k$, we pivot our sequence to a perfect sequence. See Figure~\ref{fig:pivot} for an example of these two operations when $k=3$.
\begin{figure}[h!]
    \centering
\begin{tikzpicture}[scale=0.6]
        \tri[(-6, 0)]{0,2,0,0,0,0,0,0}{violet!50};
        \tri[(-6, 0)]{0,0,0,4,4,0,0,0}{teal!50};
        \tri[(-6, 0)]{0,0,0,0,0,0,7,7}{blue!50};
        \node at (-6,-.5) {(a)};
        \tri[(0,.37)]{1,0,0,0,0,0,0}{violet!50};
        \tri[(0,.37)]{0,0,3,3,0,0,0}{teal!50}
        \tri[(0,.37)]{0,0,0,0,0,6,6}{blue!50}
        \node at (0,-.5) {(b)};
        \tri[(6,0)]{1,0,0,0,0,0,0,0}{violet!50};
        \tri[(6,0)]{0,0,3,3,0,0,0,0}{teal!50};
        \tri[(6,0)]{0,0,0,0,0,6,6,6}{blue!50};
        \node at (6,-.5) {(c)};
\end{tikzpicture}
    \caption{Pivoting to a perfect sequence when $n=8$. (a) The original sequence with weight 24. (b) Remove the bottom row. (c) Add a column continuing the last block, resulting in a sequence with weight 25.}
    \label{fig:pivot}
\end{figure}
We claim the net change in weight in this process is one. In the removal of the bottom row, we lose $\T_{k-1}+k-1=\T_k -1$. However, we gain $\T_k$ in the addition of the last column. Thus, our original sequence was not maximum weight, and we arrive at a contradiction.

\textbf{Case 2:} Suppose $|b|=k+1$ and $b_{k+1}=k+2$. This case parallels Case 1 closely: we claim $b_i\geq i$ for all $i$ using a similar argument and contradicting Proposition~\ref{prop:alternate} part (b). We use this to bound $n$, shown below
\[n=\beta+\sum_{i=1}^{k+1}b_i=\beta+\sum_{i=1}^k b_i+k+2\geq k +\sum_{i=1}^k i+ k+2=\beta+\T_{k+1}+1.\]
Note, by Lemma~\ref{lem:alg}, $\beta\geq k$ in this case, which means $n\geq \T_{k+2}-1$, which contradicts $\row{n}=k$. 

Thus, in all possible cases, the last entry in the block sequence is either $k$ or $k+1$.
\end{proof}

We note that the remaining four cases described in the proof of Proposition~\ref{prop:large_block} are realized for $n\leq 7$.

% \todo{Do we need to state these four cases? Or are we good without explicitly saying this?}
% The proof of Proposition~\ref{prop:large_block} eliminates two cases, and we  note that  the remaining four cases are realized in the following examples. 
% \begin{itemize}
% \item Case $|b|=k$, $b_k=k$: when $k=2$ and $n=5$, we have maximum weight block sequence $b=(2,2)$ for sequence $s=(1,1,0,4,4)$.
% \item Case $|b|=k$, $b_k=k+1$: when $k=2$ and $n=6$, we have maximum weight block sequence $b=(2,3)$ for sequence $s=(1,1,0,4,4,4)$
% \item Case $|b|=k+1$, $b_{k+1}=k$: when $k=2$ and $n=6$, we have maximum weight block sequence $b=(1,1,2)$ for sequence $s=(1,0,3,0,5,5)$.
% \item Case $|b|=k+1$, $b_{k+1}=k+1$: when $k=2$ and $n=7$, we have maximum weight block sequence $b=(1,1,3)$ for sequence $s=(1,0,3,0,5,5,5)$.
% \end{itemize}

%% file: results.tex
\section{Results}
\label{sec:results}

In the previous section, we developed restrictions on the lengths and entries of a maximum weight block sequence. We are now ready to prove our two main theorems. We prove Theorem~\ref{thrm:unique} for perfect sequences in Section~\ref{sec:proof-perfect}. In Section~\ref{sec:proof-correct}, we prove the more general Theorem~\ref{thrm:correct} by strong induction. In Section~\ref{sec:oeis_ver}, we show that the sequences for $|\maxweightset{n}|$ and $\maxweight{n}$ match their claimed OEIS entries. Finally, Section~\ref{sec:construct} complements Theorem~\ref{thrm:correct} by providing a direct construction that creates the maximum weight block sequences.

\subsection{Proof of Theorem~\ref{thrm:unique}}
\label{sec:proof-perfect}

We return to the interesting case of perfect sequences. We are now ready to justify this moniker by proving Theorem~\ref{thrm:unique}, which states that when $n=\T_{k+1}-2$, the unique sequence in $\maxweightset{n}$ is the perfect sequence $(1,0,3,3,0,6,6,6,0, \ldots, 0, \T_k, \T_k \ldots, \T_k)$.

%\todo{I think that the proof of the next theorem is {\bf independent} of the proof of Proposition~\ref{prop:large_block}. So could we prove this first, and simplify the previous proof by restricting ourselves to $\T_{k+1}-1 \leq n \leq \T_{k+2}-3$?}

\begin{theorem*}[Theorem~\ref{thrm:unique}]
    Let $n=\T_{k+1}-2$ for some $k\geq 1$. Then $|\maxweightset{n}|=1$, and the unique maximum weight jump-float sequence is the perfect sequence for $n$ with weight $\sum_{i=1}^k i\T_i$.
\end{theorem*}

\begin{proof}[Proof of Theorem~\ref{thrm:unique}]
Let $n=\T_{k+1}-2$, and suppose $s \in \maxweightset{n}$. 
Note that $s$ has $k$ blocks by Corollary~\ref{cor:perfect_blocks}.
We now show that the claimed block sequence is in fact unique. By $s$ being a maximum weight jump-float sequence with $k$ blocks, we know there are either $k-1$ or $k$ zeros in $s$ by Lemma~\ref{lem:alg}.

\textbf{Case 1:} There are $k-1$ zeros in $s$. By Proposition~\ref{prop:block_index}, we know $b_i=i+e_i$ where $e_i \in \{-1,0,1\}$ for all $i$. Algebraic manipulation yields
\begin{align*}
n & = k-1 +\sum_{i=1}^k b_i \\ 
\T_{k+1}-2 & =k+1-2 +\sum_{i=1}^k i+\sum_{i=1}^k e_i \\
\T_{k+1}-2 & = \T_{k+1}-2 +\sum_{i=1}^k e_i \\
0 &= \sum_{i=1}^k e_i.
\end{align*}
We claim that $e_i=0$ for all $i$, and proceed by contradiction. Assuming that some $e_i$ are nonzero, then there exist indices $i,j$ such that $e_i=1$, $e_j=-1$, and $e_\ell =0$ for all indices between these locations.  Consider the case if $i<j$. Then $b_i=b_{i+1}=i+1$ and $b_j=b_{j-1}=j-1$ and $b_\ell=\ell$ for all $i+1\leq \ell \leq j-1$. However, this contradicts Proposition~\ref{prop:alternate} part (a). A parallel argument exists if $j<i$, contradicting Proposition~\ref{prop:alternate} part (b). As such, $e_i=0$ for every index. This implies that the block sequence is unique: $b=(1,2,\ldots k)$.
 
\textbf{Case 2:} There are $k$ zeros in $s$, meaning $s_1=0$. By Proposition~\ref{prop:block_index}, we know $b_1 \in \{1,2\}$. 

If $b_1=2$, we pivot to a new jump-float sequence $s'$ with block sequence $b'=(b'_1,\ldots b'_{k+1})$, where 
\[
b'_j=\left\{\begin{array}{c l} 1 & \text{for } j=1,2, \\ b_{j-1} & \text{for } 3 \leq j \leq k+1. \end{array}\right.
\] 
See Figure~\ref{fig:unique_pivot_one} to see the first few entries of $s$ and $s'$.
We see the weight of $s'$ is the same as the weight of $s$, but this contradicts Corollary~\ref{cor:perfect_blocks}, because $s'$ has $k+1$ blocks.
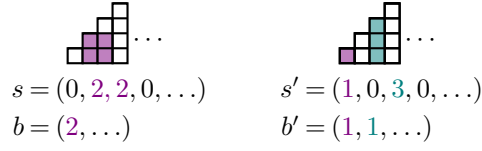
\begin{figure}[ht!]
    \centering
\begin{tikzpicture}[scale=0.6]
        \tri[(-3, -0.5)]{0,2,2,0}{violet!50};
        \node at (-2,0) {$\cdots$};
        \node (eq) at (-3,-1.5) {\setlength{\arraycolsep}{1pt} \renewcommand{\arraystretch}{1.2}$\begin{array}{rcl} s&=&( 0,\textcolor{violet}{2,2},0,\ldots) \\ b&=&(\textcolor{violet}{2},\ldots) \end{array}$};

        \tri[(3,-0.5)]{1,0,0,0}{violet!50};
        \tri[(3,-0.5)]{0,0,3,0}{teal!50};
        \node at (4,0) {$\cdots$};
        \node (eq) at (3,-1.5) {\setlength{\arraycolsep}{1pt} \renewcommand{\arraystretch}{1.2}$\begin{array}{rcl} s'&=&(\textcolor{violet}{1},0,\textcolor{teal}{3},0,\ldots ) \\ b'&=&(\textcolor{violet}{1},\textcolor{teal}{1},\ldots) \end{array}$};

\end{tikzpicture}
    \caption{Pivoting $s$ to $s'$ by breaking up the first block. Note that the weight does not change.}
    \label{fig:unique_pivot_one}
\end{figure}

If $b_1=1$, we pivot to a new jump-float sequence $s''$ with block sequence $b''=(b''_1,\ldots b''_{k})$, where
\[
b''_j=\left\{\begin{array}{c l} 2 & \text{for } j=1, \\ b_{j} & \text{for } 2 \leq j \leq k. \end{array}\right.
\] 
See Figure~\ref{fig:unique_pivot_two} to see the first few entries of $s$ and $s''$.
We see the weight of $s''$ is the same as the weight of $s$, and that $s''$ has $k-1$ zeros. However, this contradicts our work in Case 1, because $s''$ is not the proven construction. As such, $s$ cannot be maximum weight. 

\begin{figure}[ht!]
    \centering
\begin{tikzpicture}[scale=0.6]
        \tri[(-3, -0.5)]{0,2,0}{violet!50};
        \node at (-2.1,0) {$\cdots$};
        \node (eq) at (-3,-1.5) {\setlength{\arraycolsep}{1pt} \renewcommand{\arraystretch}{1.2}$\begin{array}{rcl} s&=&(0, \textcolor{violet}{2},0, \ldots) \\ b&=&(\textcolor{violet}{1},\ldots)
        \end{array}$};

        \tri[(3,-0.5)]{1,1,0}{violet!50};
        \node at (3.9,0) {$\cdots$};        
        \node (eq) at (3,-1.5) {\setlength{\arraycolsep}{1pt} \renewcommand{\arraystretch}{1.2}$\begin{array}{rcl} s''&=&(\textcolor{violet}{1,1},0,\ldots)  \\ b''&=&(\textcolor{violet}{2},\ldots)
        \end{array}$};

\end{tikzpicture}
    \caption{Pivoting $s$ to $s''$ by widening the first block. Note that the weight does not change.}
    \label{fig:unique_pivot_two}
\end{figure}
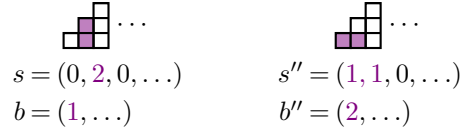

Thus, in all cases, the maximum weight jump-float sequence is as claimed. Further, it is unique. Finally, we note that the weight of this sequence immediately follows by definition, and we refer back to Definition~\ref{def:perfect} for a closed form. 
\end{proof}

\subsection{Proof of Theorem~\ref{thrm:correct}}

\label{sec:proof-correct}

With the structural groundwork set, we are now ready to prove our main result about the enumeration and weight of the maximum weight jump-float sequences.

\begin{theorem*}[Theorem~\ref{thrm:correct}]
    For $n \geq 1$, the number of maximum weight sequences is
\begin{equation}
\label{eqn:correct-count}
     |\maxweightset{n}|
    =
    \binom{\row{n}}{\ind{n}}+\binom{\row{n}+1}{\ind{n}-1}  
\end{equation}

    with weight
\begin{equation}
\label{eqn:correct-weight}
     \maxweight{n}=\T_{n-\row{n}} + \sum_{i=1}^{\row{n}}  \T_i
     =\sum_{j=1}^{n-\row{n}} j + \sum_{i=1}^{\row{n}}  \T_i,
\end{equation}
    where $\T_i = {i+1 \choose 2}$ is the $i$th triangular number, and $\rho_n$ and $\lambda_n$ are defined in Lemma~\ref{lem:row-ind}.
\end{theorem*}

\begin{proof}
We prove the theorem by strong induction on the size $n$. As per Lemma~\ref{lem:row-ind}, let $\row{n}$ and $\ind{n}$ be the unique positive integers such that
$$
n = \T_{\row{n} +1} - 2 + \ind{n}.
$$
where $\T_j = {j+1 \choose 2}$ is the $j$th triangular number and $0 \leq \ind{n} \leq \row{n}+1$. It is straightforward to check that the theorem holds for $1 \leq n < \T_3 -2 = 4$, see Figure~\ref{fig:main-base-case}.

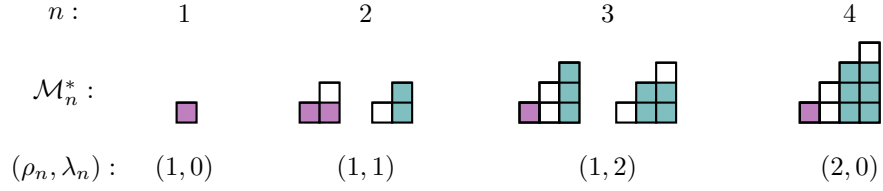
\begin{figure}[ht!]
\centering
\begin{tikzpicture}[scale=0.8]

\begin{scope}[shift={(-2,0)}]
    \node at (0,1.8) {$n:$};
    \node at (0,0.5) {$\maxweightset{n}:$};
    \node at (0,-0.75) {$(\row{n}, \ind{n}):$};
\end{scope}

\begin{scope}[shift={(0,0)}]
    \node at (0,1.8) {$1$};
    \tri[(0.2, 0)]{1}{violet!50};
    \node at (0,-0.75) {$(1, 0)$};
\end{scope}

\begin{scope}[shift={(3,0)}]
\node at (0,1.8) {$2$};
 \tri[(-0.6, 0)]{1,1}{violet!50};   
 \tri[(0.6, 0)]{0,2}{teal!50}; 
     \node at (0,-0.75) {$(1, 1)$};
\end{scope}

\begin{scope}[shift={(7,0)}]
\node at (0,1.8) {$3$};
\tri[(-.8, 0)]{1,0,0}{violet!50};
\tri[(-.8, 0)]{0,0,3}{teal!50};
\tri[(.8, 0)]{0,2,2}{teal!50};  
    \node at (0,-0.75) {$(1, 2)$};
\end{scope}

\begin{scope}[shift={(11,0)}]
\node at (0,1.8) {$4$};
 \tri[(0, 0)]{1,0,0,0}{violet!50};
 \tri[(0, 0)]{0,0,3,3}{teal!50};
     \node at (0,-0.75) {$(2, 0)$};
\end{scope} 
\end{tikzpicture}

\caption{The base cases $1 \leq n \leq 4$ for Theorem~\ref{thrm:correct}. Checking that equations \eqref{eqn:correct-count} and \eqref{eqn:correct-weight} both hold is elementary.}
\label{fig:main-base-case}
\end{figure}

First, let us consider the cases when $\ind{n}=0$, so that $n= \T_{\row{n}+1} - 2$. In other words, we restrict ourselves to the perfect sequences corresponding to $n \in \{ \T_{k+1} -2 : k \in \mathbb{N} \}$, and proceed by induction on $k := \row{n}$.
By Theorem~\ref{thrm:unique}, 
we have $|\maxweightset{n}|=1$, and equation~\eqref{eqn:correct-count} becomes
$$
{k \choose 0} + {k+1 \choose -1} = 1.
$$
Next, Theorem~\ref{thrm:unique} tells us that $\maxweight{n} = \sum_{i=1}^{k} i \T_i$, and we must show that this equals the value of the weight equation \eqref{eqn:correct-weight}. 
Assume that this weight equation holds for $n=\T_{k+1}-2$, namely
$$
\sum_{i=1}^k i \T_i
%= \T_{\T_{k+1}-2-k} + \sum_{i=1}^k \T_i
= \sum_{j=1}^{\T_{k+1}-2-k} j + \sum_{i=1}^k \T_i.
$$
Now considering equation~\eqref{eqn:correct-weight} for $n=\T_{k+2}-2$, we have
\begin{align*}
%\T_{\T_{k+2} -2 - (k+1)} + \sum_{i=1}^{k+1} \T_i 
%&=
\sum_{j=1}^{\T_{k+2} -2 -(k+1)} j  +  \sum_{i=1}^{k+1} \T_i  
 &=  \left(  \sum_{j=1}^{\T_{k+1}-2-k} j + \sum_{i=1}^{k}\T_i  \right)
 + \left( \T_{k+1} + \sum_{j=\T_{k+1} -2- k+1}^{\T_{k+2} -2 -(k+1)} j  \right) \\
 &= \sum_{i=1}^{k} i \T_i + \left(
 \T_{k+1} + (\T_{k+1} -k-2)(k+1) + 
 \sum_{j=1}^{k+1} j  \right) \\
 %&= \sum_{i=1}^{k} i \T_i + \left(
 %\T_{k+1} + (\T_{k+1} -k-2)(k+1) + 
 %\T_{k+1} \right) \\ 
  &= \sum_{i=1}^{k} i \T_i + 
 (k+3) \T_{k+1} - (k+2)(k+1)    \\ 
   &= \sum_{i=1}^{k} i \T_i + 
 (k+1) \T_{k+1}  \quad = \quad  \sum_{i=1}^{k+1} i \T_i,
\end{align*}  
where we used induction in the second line. So both the count formula and the weight formula are correct for
$n \in \{ \T_{k+1} -2 : k \in \mathbb{N} \}$.

Next, we consider the cases when $\ind{n} \neq 0$. Consider a sequence $s \in \maxweightset{n}$ where
$n = \T_{\row{n}+1}-2 + \lambda_n$. By Proposition~\ref{prop:large_block}, the final entry of the block sequence is either $\row{n}$ or $\row{n}+1$.

{\bf Case 1:} the final block sequence entry is $\row{n}$. Let $s'$ be the sequence obtained by removing the last $\row{n}+1$ entries of $s$. Since $s \in \maxweightset{n}$, we must have $s' \in \maxweightset{n'}$ where 
$$
n' = n - (\row{n}+1) = \T_{\row{n}+1}-2 + \lambda_n - (\row{n}+1) = \T_{\row{n}}-2 + \lambda_n.
$$
By induction, there are
$$
|\maxweightset{n'}| = {\row{n}-1 \choose \ind{n}} +
{\row{n} \choose \ind{n}-1}
$$
choices for subsequence $s'$. Furthermore, the weight of $s$ is
\begin{align}
\nonumber
w(s) &= \row{n}(n-\row{n}+1) + w(s')  = \row{n}(n-\row{n}+1) + \maxweight{n-(\row{n}+1)}  \\
\label{eqn:weight-first-part-k}
&=
\row{n}(n-\row{n}+1) + \sum_{j=1}^{n-2\row{n}} j  + \sum_{i=1}^{\row{n}-1} \T_i 
\end{align}
which expands and simplifies to 
\begin{align*}
 & \quad \sum_{i=1}^{\row{n}} \T_i - \frac{\row{n}(\row{n}+1)}{2} + \row{n}(n-\row{n}+1) + \sum_{j=1}^{n-2\row{n}} j \\
  &= \sum_{i=1}^{\row{n}} \T_i + \row{n}(n-2\row{n}+1) + \frac{\row{n}(\row{n}-1)}{2} + \sum_{j=1}^{n-2\row{n}} j \\
  &= \sum_{i=1}^{\row{n}} \T_i + \row{n}(n-2\row{n}+1) + \sum_{j=0}^{\row{n}-1} j + \sum_{j=1}^{n-2\row{n}} j \\
  &= \sum_{i=1}^{\row{n}} \T_i + \sum_{j=n-2\row{n}+1}^{n-\row{n}} j + \sum_{j=1}^{n-2\row{n}} j \\  
  &= \sum_{i=1}^{\row{n}} \T_i + \sum_{j=1}^{n-\row{n}} j,  
\end{align*} 
as required for equation~\eqref{eqn:correct-weight}.

{\bf Case 2:} the final block sequence entry is $\row{n}+1$. Let $s''$ be the sequence obtained by removing the last $\row{n}+2$ entries of $s$. Since $s \in \maxweightset{n}$, we must have $s'' \in \maxweightset{n''}$ where 
$$
n'' = n - (\row{n}+2) = \T_{\row{n}+1}-2 + \lambda_n - (\row{n}+2) = \T_{\row{n}}-2 + (\lambda_n -1).
$$
By induction, there are
$$
|\maxweightset{n''}| = {\row{n}-1 \choose \ind{n}-1} +
{\row{n} \choose \ind{n}-2}
$$
choices for subsequence $s''$. We show that the weight of a sequence in this case 
\begin{align}
\nonumber
 w(s) 
 &= (\row{n}+1)(n-\row{n})+w(s'') = (\row{n}+1)(n-\row{n})+\maxweight{n-(\row{n}+2)}   \\
\label{eqn:weight-first-part-k-plus-one}
 &= (\row{n}+1)(n-\row{n}) +  \sum_{i=1}^{\row{n}-1}  \T_i + \sum_{j=1}^{n-2\row{n}-1} j 
\end{align}
equals the weight of a sequence in the first case.
Subtracting equation~\eqref{eqn:weight-first-part-k-plus-one} from equation~\eqref{eqn:weight-first-part-k} yields
$$
\row{n} (n- \row{n}+1) + n- 2\row{n} - (\row{n}+1) (n- \row{n}) = 0.
$$

Since the weights of the sequences constructed by these two cases are equal, we have shown that
\begin{align*}
    |\maxweightset{n}| = 
    |\maxweightset{n'}| + |\maxweightset{n''}| 
    &=
    {\row{n}-1 \choose \ind{n}} + {\row{n} \choose \ind{n}-1} 
    + 
    {\row{n}-1 \choose \ind{n}-1} + {\row{n} \choose \ind{n}-2} \\
    &=
    {\row{n} \choose \ind{n}} + {\row{n}+1 \choose \ind{n}-1}
\end{align*}
as desired. 
\end{proof}

\subsection{OEIS sequence verification}
\label{sec:oeis_ver}

We now have closed formulas for the sequences $|\maxweightset{n}|$ and $\maxweight{n}$. In this section, we confirm the claims made in Section~\ref{sec:intro} associating $|\mathcal{Z}_{n+1}(\pattern)| = |\maxweightset{n}|$ to OEIS sequence A209561 and $\maxinv \permset{n+1}(\pattern) = \maxweight{n}$ to OEIS sequence A023536.

From the proof of Theorem~\ref{thrm:correct}, note that the equality \[|\maxweightset{n}| = |\maxweightset{n'}| + |\maxweightset{n''}|\] 
is the Pascal-like recurrence that we observed in trapezoid \eqref{eqn:K-trapezoid}. This confirms that our sequence $\{ | \maxweightset{n} | \}$ matches OEIS sequence A209561, since the entries of that sequence also adhere to this rule. 

The proof for $\maxweight{n}$ is algebraic.

\begin{theorem}
    \label{thrm:oeis_weight}
    The sequence $\maxweight{n}$ defined in Theorem~\ref{thrm:correct} matches OEIS sequence A023536.
\end{theorem}

\begin{proof}
    To prove the claim, we begin with a formula stated for sequence A023536:
    \[a_n=\frac{(n+1)(n+2)}{2}-\sum_{j=1}^{n+1}\left\lfloor \frac{\sqrt{8j+1}-1}{2}\right\rfloor.\]

    We define $\kappa_j=\left\lfloor \frac{\sqrt{8j+1}-1}{2}\right\rfloor$, and straightforward algebra shows $\kappa_j=\max\{i : \T_i \leq j\}$. For any integer $m$, there are $m+1$ values $j$ such that $\kappa_j=m$, specifically $\{\T_m, \T_{m}+1,\ldots, \T_{m+1}-1\}$. 

    With this understanding of $\kappa_j$, we aim to rewrite the sum in the formula for $a_n$ without $\kappa_j$ in the summand. Every term $k$ will appear $k+1$ times, except for possibly the largest value, which appears $n+1-(\T_{\kappa_{n+1}}-1)$ times; specifically, for $j$ values $\{\T_{\kappa_{n+1}},\T_{\kappa_{n+1}}+1,\ldots, n+1\}$. As such, we rewrite $a_n$ as
    \[a_n = \frac{(n+1)(n+2)}{2}- \sum_{k=1}^{\kappa_{n+1}-1}k(k+1) - (\kappa_{n+1})(n+1-(\T_{\kappa_{n+1}}-1)).\]

    Here we begin pivoting to our formula for $\maxweight{n}$. A change in variables yields 
    $\kappa_{n+1}=\row{n-1}+1$ which we use to rewrite $a_n$ as
    \[a_n = \frac{(n+1)(n+2)}{2}-\sum_{k=1}^{\row{n-1}}k(k+1) - (\row{n-1}+1)(n+2-\T_{\row{n-1}+1}).\]

    We claim the formula can be rewritten as 
    \[a_n = \frac{(n+1)(n+2)}{2}-\sum_{k=1}^{\row{n}}k(k+1) - (\row{n}+1)(n+2-\T_{\row{n}+1}).\] When $n\neq \T_{\row{n}+1}-2$, the formula is immediate, as $\row{n}=\row{n-1}$. When $n=\T_{\row{n}+1}-2$, we know $\row{n}=\row{n-1}+1$, and as such
    \begin{align*}
        \sum_{k=1}^{\row{n-1}}k(k+1)& + (\row{n-1}+1)(n+2-\T_{\row{n-1}+1}) 
         = \sum_{k=1}^{\row{n}-1}k(k+1) + (\row{n})(\T_{\row{n}+1}-\T_{\row{n-1}+1}) 
        \\ & = \sum_{k=1}^{\row{n}-1}k(k+1) +(\row{n})(\row{n}+1) = \sum_{k=1}^{\row{n}}k(k+1) 
        \\ &= \sum_{k=1}^{\row{n}}k(k+1) + (\row{n}+1)(n+2-\T_{\row{n}+1})
    \end{align*}
    where the final equality holds because the added term is zero. Thus, the claimed formula for $a_n$ works in both cases.

    What remains is routine, if careful, algebra:
    \begin{align*}
        a_n & = \frac{(n+1)(n+2)}{2}-\sum_{k=1}^{\row{n}}k(k+1) - (\row{n}+1)(n+2-\T_{\row{n}+1}) \\ 
        & = \frac{(n+1)(n+2)}{2} - 2 \binom{\row{n}+2}{3} - (\row{n}+1)\left(n+2-\frac{(\row{n}+2)(\row{n}+1)}{2}\right) \\
        & = \frac{n^2}{2}-n\row{n}+\frac{n}{2}+\frac{\row{n}^3}{6}+\row{n}^2-\frac{\row{n}}{6} \\
        %& = \frac{n^2}{2}-\frac{2n\row{n}}{2}+\frac{\row{n}^2}{2}+\frac{n}{2}-\frac{\row{n}}{2} + \frac{\row{n}^3}{6}+\frac{3\row{n}^2}{6}+\frac{2\row{n}}{6} \\
        & = \frac{(n-\row{n}+1)(n-\row{n})}{2} + \frac{\row{n}(\row{n}+1)(\row{n}+2)}{6}\\ 
        & = \binom{n-\row{n}+1}{2} + \binom{\row{n}+2}{3}  \\
        & = \T_{n-\row{n}} + \sum_{i=1}^{\row{n}} \T_i = \maxweight{n}
    \end{align*}
    where the formula $\sum_{i=1}^{\row{n}} \T_i = \binom{\row{n}+2}{3} $ is from OEIS sequence A000292 \cite{oeis}.
\end{proof}

We conclude this section with an intuitive interpretation of the maximum weight formula $\maxweight{n}=\T_{n-\row{n}} + \sum_{i=1}^{\row{n}}  \T_i$. Setting $\maxweight{0}=0$ for convenience,  consider the difference sequence
$$
\{ d_n \}_{n \geq 1} := \{ \maxweight{n} - \maxweight{n-1} \}_{n \geq 1} = 1,1,2,3,3,4,5,6,6,7,8,9,10,10,11, 12, 13,14, \ldots.
$$
It is clear that $\maxweight{n} = \sum_{j=1}^n d_n$.  
The $d_j$ terms for $j=\T_{k+1}-2$ (which correspond to the perfect sequences) constitute the $\sum_{i=1}^{\row{n}}  \T_i$ summand. The remaining terms constitute the $\T_{n-\row{n}} = \sum_{i=1}^{n-\row{n}} i$ summand.

\subsection{Construction of maximum weight jump-float sequences}
\label{sec:construct}

%With the enumeration of $\maxweightset{n}$ in hand, we now give a construction method for their block sequences.

In the previous section, we enumerated $\maxweightset{n}$ via an inductive proof that uses the Pascal-like recurrence of trapezoid \eqref{eqn:K-trapezoid}. We now complement that result by offering a constructive method that gives explicit meaning to the formula $|\maxweightset{n}| = {\row{n} \choose \ind{n}} + {\row{n}+1 \choose \ind{n}-1}$.
%This second constructive method alters the two nearby perfect sequences to obtain the sequences of $\maxweightset{n}$.
This method creates all of the sequences of $\maxweightset{n}$, located at $(\row{n}, \ind{n})$, by either incrementing exactly $\ind{n}$ blocks, each by one, of the perfect sequence located at $(\row{n},0)$, or decrementing exactly $\row{n}-\ind{n}+2$ blocks, each by one, of the perfect sequence located at $(\row{n}+1,0)$. These two disjoint cases, along with the number of ways to choose which blocks, naturally explain the proven formula for $|\maxweightset{n}|$. 

See Figure~\ref{fig:construction_example2} for an example of this process where $n=5$. Note that $\row{n}=2$ and $\ind{n}=1$, and the perfect sequences considered at  $(\row{n},0)$ and  $(\row{n}+1,0)$ are lengths four and eight, respectively. 

\begin{figure}[h!]
\label{fig:construction_example2}
    \centering
    \begin{tikzpicture}[scale=0.6]
    \node at (-4,4) {\begin{tabular}{c} \underline{Construction A:} \\ Add 1 column to \\ the triangle for $n=4$ \end{tabular}};
    \tri[(-4,1)]{1,0,0,0}{violet!50}; 
    \tri[(-4,1)]{0,0,3,3}{teal!50}; 
    \node (las) at (-4,1) {};

    \tri[(-5.75,-3)]{1,1,0,0,0}{violet!50};
    \tri[(-5.75,-3)]{0,0,0,4,4}{teal!50};
    \node at (-6.8,-4) {\begin{tabular}{r} Add 1 column \\ to first block \end{tabular}};
    \node (lae1) at (-6,-1.5) {};

    \tri[(-2,-3)]{1,0,0,0,0}{violet!50};
    \tri[(-2,-3)]{0,0,3,3,3}{teal!50};
    \node at (-1.3,-4) {\begin{tabular}{l} Add 1 column \\ to second block \end{tabular}};
    \node (lae2) at (-2,-1.5) {};

    \draw[->] (las) -- (lae1);
    \draw[->] (las) -- (lae2);

    \node at (7,4) {\begin{tabular}{c} \underline{Construction B:} \\ Remove 3 columns \\ from the triangle for $n=8$ \end{tabular}};
    \tri[(7,-0.5)]{1,0,0,0,0,0,0,0}{violet!50};
    \tri[(7,-0.5)]{0,0,3,3,0,0,0,0}{teal!50};
    \tri[(7,-0.5)]{0,0,0,0,0,6,6,6}{blue!50};
    \node (ras) at (7,-0.5) {};

    \tri[(7,-3)]{0,2,0,0,0}{teal!50};
    \tri[(7,-3)]{0,0,0,4,4}{blue!50};
    \node at (7,-4) {\begin{tabular}{c} Remove 1 column \\ from each block \end{tabular}};
    \node (rae) at (7,-1.5) {};

    \draw[->] (ras) -- (rae);
    
\end{tikzpicture}
    \caption{Building maximum weight jump-float sequences of length 5 by adding one column to a block of the perfect sequence of length 4 and removing three columns from the perfect sequence of length 8. These operations are labeled Construction A and B to parallel Theorem~\ref{thrm:structure}.}
\end{figure}
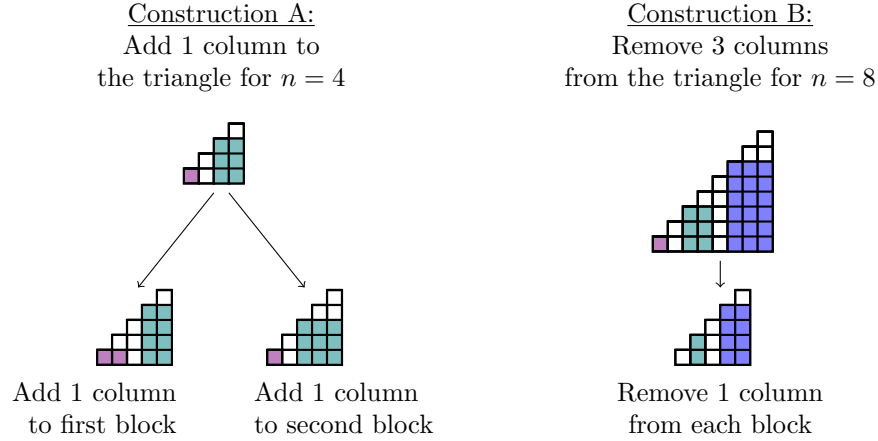

%Indeed, given $n= \T_{\row{n}+1}-2 + \ind{n}$, the two nearest perfect sequences appear at $n^- = \T_{\row{n}+1}-2$ and $n^+ = \T_{\row{n}+2}-2$. The perfect sequence $p^-$ for $n^-$ has $\row{n}$ blocks and the perfect sequence $p^+$ for $n^+$ has $\row{n}+1$ blocks. We will see that the sequences of $\maxweightset{n}$ are constructed by either incrementing $\ind{n}$ blocks of $p^-$, or decrementing $\row{n}-\ind{n}+2$ blocks of $p^+$.  There are ${\row{n} \choose \ind{n}}$ ways to do the former, and ${\row{n}+1 \choose \row{n}+2-\ind{n}} = {\row{n}+1 \choose \ind{n}-1}$ to do the latter. 
This constructive method explicitly identifies when a block sequence has a leading gap. Therefore, in order to proceed, we must define an \emph{extended block sequence}.

\begin{definition}
    Let $s = (s_1, s_2, \ldots, s_n) \in \maxweightset{n}$ with block sequence $b = (b_1, b_2, \ldots, b_k) \in \mwblockset{n}$. The \emph{extended block sequence} $\overline{b}$ for $s$ is defined as follows.
    If $s_1=1$ then  $\overline{b} = b$. 
    If $s_1=0$ then $\overline{b} = 
    (\overline{b}_1, \overline{b}_2, \overline{b}_3, \ldots, \overline{b}_{k+1}) := 
    (0, b_1, b_2, \ldots, b_k)$.
    The set of extended block sequences is denoted $\extblockset{n}$.
\end{definition}

% In other words, the extended block $\overline{b}$ sequence prepends a zero to block sequence $b$ when the corresponding jump-float sequence $s$ has a leading gap $s_1=0$. The necessity of the extended block sequence is highlighted in Figure~\ref{fig:block}. Note that these two sequences of different lengths yield the same block sequence, but their extended block sequences differentiate them. In fact, extended block sequences are in one-to-one correspondence with maximum weight jump-float sequences.
% \todo{If we keep my block sequence changes, then the ``necessity" isn't to distinguish these two cases. The necessity comes from the constructive method itself, which could end up with $b_1=0$.}

See Figure~\ref{fig:block} for an example of the extended block sequence of two jump-float sequences.

\begin{figure}[h!]
\centering
\begin{tikzpicture}[scale=0.6]

\tri[(-2,0)]{1,0,0,0,0,0,0,0,0}{violet!50};
\tri[(-2,0)]{0,0,3,3,3,0,0,0,0}{teal!50};
\tri[(-2,0)]{0,0,0,0,0,0,7,7,7}{blue!50};

\node (eq) at (-2.2,-1.2) {\setlength{\arraycolsep}{1pt} \renewcommand{\arraystretch}{1.2}$\begin{array}{rcl} s&=&(\textcolor{violet}{1}, 0, \textcolor{teal}{3, 3, 3}, 0, \textcolor{blue}{7,7, 7}) \\ \overline{b}&=&(\textcolor{violet}{1}, \textcolor{teal}{3}, \textcolor{blue}{3})
\end{array}$};

\end{tikzpicture}
\hspace{.2 in}
\begin{tikzpicture}[scale=0.6]

\tri[(-2,0)]{0,2,0,0,0,0,0,0,0}{violet!50};
\tri[(-2,0)]{0,0,0,4,4,0,0,0,0}{teal!50};
\tri[(-2,0)]{0,0,0,0,0,0,7,7,7}{blue!50};

\node (eq) at (-2.2,-1.2) {\setlength{\arraycolsep}{1pt} \renewcommand{\arraystretch}{1.2}$\begin{array}{rcl} s&=&(0,\textcolor{violet}{2}, 0, \textcolor{teal}{4, 4}, 0, \textcolor{blue}{7, 7, 7}) \\ \overline{b}&=&(0,\textcolor{violet}{1}, \textcolor{teal}{2}, \textcolor{blue}{3})
\end{array}$};

\end{tikzpicture}

\caption{Two jump-float sequences from $\maxweightset{9}$ with their triangle representations and their extended block sequences.}
\label{fig:block}
\end{figure}
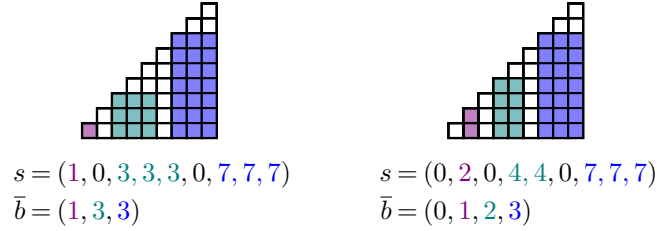

The extended block sequence of a jump-float sequence will follow the same sorts of rules as the block sequence in terms of its entries, and we highlight that the last entry of $\overline{b}$ is either $\row{n}$ or $\row{n}+1$ by Proposition~\ref{prop:large_block}. We are now ready to state the main result of this section, the construction of all extended block sequences (which also constructs the set of maximum weight jump-float sequences). 

\begin{theorem}
\label{thrm:structure}
For $n \in \mathbb{N}$, let $(\row{n},\ind{n})$ be its index in size trapezoid $S$ where $\row{n} \geq 1$ and $0 \leq \ind{n} \leq \row{n}+1$. The extended block sequence $\overline{b}  \in \extblockset{n}$ for $s \in \maxweightset{n}$ has one of the following forms:
\begin{enumerate}
\item[A.] The sequence $s$ has extended block sequence $(\overline{b}_1, \overline{b}_{2}, \ldots, \overline{b}_{\row{n}})$ where
\[\overline{b}_j = \left\{\begin{array}{c l} j & \text{for } j \not\in X
\\j+1 & \text{for } j \in X \end{array} \right. \quad \mbox{where} \quad X \in {[\row{n}] \choose \ind{n}}.\]

\item[B.] The sequence $s$ has extended block sequence $(\overline{b}_{1},\overline{b}_2, \ldots, \overline{b}_{\row{n}+1})$ where
\[\overline{b}_j = \left\{\begin{array}{c l} j & \text{for } j \in X
\\j-1 & \text{for } j \not\in X \end{array} \right. \quad \mbox{where} \quad X \in {[\row{n}+1] \choose \ind{n}-1}\]
\end{enumerate}
\end{theorem}

We note that in the statement of Theorem~\ref{thrm:structure}, and subsequently in the inductive step of the proof of the theorem, the range of $\ind{n}$ will allow for $\binom{\row{n}}{\row{n}+1}$ and $\binom{\row{n+1}}{-1}$. This is in fact correct, and it implies that for certain values of $n$, all sequences are built only from Construction B or A, respectively. 

\begin{proof}[Proof of Theorem~\ref{thrm:structure}]
We will proceed by strong induction on $n$, the length of the jump-float sequence. Note that for $1 \leq n \leq 4$, the construction holds by direct computation, seen in Figure~\ref{fig:struc-base-case}.

\begin{figure}[ht!]
\centering
\begin{tikzpicture}[scale=0.8]

\begin{scope}[shift={(-3.5,0)}]
    \node at (0,1.8) {$n:$};
    \node at (0,0.5) {$\maxweightset{n}:$};
    \node at (0,-0.75) {Construction:};
    \node at (0, -1.5) {$X$:};
    \node at (0, -2.25) {$\overline{b}$ :};

%     \node at (0, -1.5) {
%         \begin{tabular}{r}
% Construction: \\
% Set: \\
% Extended Base Sequence:
%         \end{tabular}
%     };
\end{scope}

\begin{scope}[shift={(0,0)}]
    \node at (0,1.8) {$1$};
    \tri[(0.2, 0)]{1}{violet!50};
    \node at (0,-0.75) {A};
    \node at (0,-1.5) {$\emptyset$};
    \node at (0,-2.25) {$(\textcolor{violet}{1})$};
\end{scope}

\begin{scope}[shift={(3,0)}]
    \node at (0,1.8) {$2$};
    \tri[(-0.6, 0)]{1,1}{violet!50};   
    \tri[(0.7, 0)]{0,2}{teal!50}; 
    \node at (-0.78,-0.75) {A};
    \node at (-0.78,-1.5) {$\{1\}$};
    \node at (-0.78,-2.25) {$(\textcolor{violet}{2})$};
    \node at (0.6,-0.75) {B};
    \node at (0.6,-1.5) {$\emptyset$};
    \node at (0.6,-2.25) {$(0,\textcolor{teal}{1})$};    

\end{scope}

\begin{scope}[shift={(7,0)}]
    \node at (0,1.8) {$3$};
    \tri[(-.8, 0)]{1,0,0}{violet!50};
    \tri[(-.8, 0)]{0,0,3}{teal!50};
    \tri[(.9, 0)]{0,2,2}{teal!50};  
    \node at (-0.85,-0.75) {B};
    \node at (-0.85,-1.5) {$\{2\}$};
    \node at (-0.85,-2.25) {$(\textcolor{violet}{1},\textcolor{teal}{1})$};
    \node at (0.85,-0.75) {B};
    \node at (0.85,-1.5) {$\{1\}$};
    \node at (0.85,-2.25) {$(0,\textcolor{teal}{2})$};    
\end{scope}

\begin{scope}[shift={(11,0)}]
    \node at (0,1.8) {$4$};
    \tri[(0, 0)]{1,0,0,0}{violet!50};
    \tri[(0, 0)]{0,0,3,3}{teal!50};
    \node at (0,-0.75) {A};
    \node at (0,-1.5) {$\emptyset$};
    \node at (0,-2.25) {$(\textcolor{violet}{1},\textcolor{teal}{2})$};
\end{scope} 
\end{tikzpicture}

\caption{The base cases $1 \leq n \leq 4$ for Theorem~\ref{thrm:structure}. 
%The case for $n=4$ has been included as another example.
}
\label{fig:struc-base-case}
\end{figure}

Assume our construction holds for sequences of all lengths less than $n$ for some $n > 4$, and consider the case where the length is $n$. Let $s$ be a maximum weight jump-float sequence of length $n$, with corresponding extended block sequence $\overline{b}$. We construct $s'$ by removing the last block (and its preceding gap) of $s$. It has corresponding extended block sequence $\overline{b}'$. By Proposition~\ref{prop:large_block}, we know the last block of $s$ is either length $\row{n}$ or $\row{n}+1$. We will consider cases based on these lengths. Note, these two cases further highlight the Pascal-like recurrence in the sequence: Case 1 is the ``upper right'' neighbor, while Case 2 is the ``upper left'' neighbor, as seen by the calculation of the index in the row for each case.

\textbf{Case 1:} If the last block of $s$ is length $\row{n}$, then $s'$ is length $n':=n-(\row{n}+1)$. Straightforward algebra shows that $\row{n'}=\row{n}-1$ and $\ind{n'}=\ind{n}$. By strong induction, we know $s'$ either has the form given by Construction A or Construction B in the theorem statement. We consider both cases.
\begin{itemize}
    \item[A.] If $s'$ follows Construction A, then we know that every block in $\overline{b}'=(\overline{b}'_1,\ldots,\overline{b}'_{\row{n'}})$ is of the form
    \[\overline{b}'_j = \left\{\begin{array}{c l} j & \text{for } j \not\in X' \\j+1 & \text{for } j \in X' \end{array} \right.\] for some $X'\in {[\row{n'}] \choose \ind{n'}}$. 
    
    We now consider the form of $\overline{b}$. Since the last block is length $\row{n}$, we can write the extended block sequence for $s$ as $\overline{b}=(\overline{b}_1,\ldots \overline{b}_{\row{n}})$, where \[\overline{b}_j = \left\{\begin{array}{c l} j & \text{for } j \not\in X
    \\j+1 & \text{for } j \in X \end{array} \right.\]
    where $X = X'$. Note, this implies $s$ follows Construction $A$.
    
    \item[B.] If $s'$ follows Construction B, then we know that every block in $\overline{b}'=(\overline{b}'_1,\ldots,\overline{b}'_{\row{n'}+1})$ is of the form
    \[\overline{b}'_j = \left\{\begin{array}{c l} j & \text{for } j \in X'\\j-1 & \text{for } j \not\in X' \end{array} \right. \]
    for some $X' \in {[\row{n'}+1] \choose \ind{n'}-1}.$

    We now consider the form of $\overline{b}$. Since the last block is length $\row{n}$, we can write the extended block sequence of $s$ as $\overline{b}=(\overline{b}_1,\ldots,\overline{b}_{\row{n}+1})$, where
    \[\overline{b}_j = \left\{\begin{array}{c l} j & \text{for } j \in X
    \\j-1 & \text{for } j \not\in X \end{array} \right. \] where $X=X'$. Note, this implies $s$ follows Construction B.

\end{itemize}

\textbf{Case 2:} If the last block of $s$ is length $\row{n}+1$, then $s'$ is length $n'=n-(\row{n}+2)$. Straightforward algebra shows that $\row{n'}=\row{n}-1$ and $\ind{n'}=\ind{n}-1.$ By strong induction, we know $s'$ either has the form given by Construction A or Construction B in the theorem statement. We consider both cases.

\begin{itemize}
    \item[A.] If $s'$ follows Construction A, then we know that every block in $\overline{b}'=(\overline{b}'_1,\ldots,\overline{b}'_{\row{n'}})$ is of the form
    \[\overline{b}'_j = \left\{\begin{array}{c l} j & \text{for } j \not\in X' \\j+1 & \text{for } j \in X' \end{array} \right.\]
    for some $X'\in {[\row{n'}] \choose \ind{n'}}.$

    We now consider the form of $\overline{b}$. Since the last block is length $\row{n}+1$, we can write the extended block sequence for $s$ as $\overline{b}=(\overline{b}_1,\ldots \overline{b}_{\row{n}})$, where \[\overline{b}_j = \left\{\begin{array}{c l} j & \text{for } j \not\in X
    \\j+1 & \text{for } j \in X \end{array} \right.\]
    where $X = X' \cup \{\row{n}\}$. Note, this implies $s$ follows Construction $A$.
    
    \item[B.] If $s'$ follows Construction B, then we know that every block in $\overline{b}'=(\overline{b}'_1,\ldots,\overline{b}'_{\row{n'}+1})$ is of the form
    \[\overline{b}'_j = \left\{\begin{array}{c l} j & \text{for } j \in X'\\j-1 & \text{for } j \not\in X' \end{array} \right. \]
    for some $X' \in {[\row{n'}+1] \choose \ind{n'}-1}.$

    We now consider the form of $\overline{b}$. Since the last block is length $\row{n}+1$, we can write the extended block sequence of $s$ as $\overline{b}=(\overline{b}_1,\ldots,\overline{b}_{\row{n}+1})$, where
    \[\overline{b}_j = \left\{\begin{array}{c l} j & \text{for } j \in X
    \\j-1 & \text{for } j \not\in X \end{array} \right. \] where $X=X' \cup \{\row{n}\}$. Note, this implies $s$ follows Construction B.
\end{itemize}

Combining the results of these four cases, we have shown that $s$ is of the form given in either Construction A or Construction B. 

We also note that by this inductive construction, we are ensured that every set $X$ in both $\binom{[\row{n}]}{\ind{n}}$ and $\binom{[\row{n}+1]}{\ind{n}-1}$ is considered. Note that Figure~\ref{fig:struc-base-case} shows the sets for $1 \leq n \leq 4$, and construction independently considers both adding and not adding $\row{n}$ and $\row{n}+1$ to the sets.
\end{proof}

{\color{blue} }

To further illustrate the previous theorem, Figure~\ref{fig:construction_example1} returns to the example considered in Figure~\ref{fig:construction_example2} for $n=5$, highlighting the sets $X$ which build the extended block sequences. We note that $\row{n}=2$ and $\ind{n}=1$. 
\begin{figure}[h!]
    \centering
    \begin{tikzpicture}[scale=0.6]
    \node at (0,2.5) {\underline{Construction A:} $X\in \binom{[\row{n}]}{\ind{n}}=\binom{[2]}{1}$ };

    \tri[(-2,0)]{1,1,0,0,0}{violet!50};
    \tri[(-2,0)]{0,0,0,4,4}{teal!50};
    \node at (-2.2,-1.2) {\setlength{\arraycolsep}{1pt} \renewcommand{\arraystretch}{1.2} $\begin{array}{rcl} X&=& \{1\} \\ \overline{b} &=& (\textcolor{violet}{2},\textcolor{teal}{2}) \end{array}$};

    \tri[(2,0)]{1,0,0,0,0}{violet!50};
    \tri[(2,0)]{0,0,3,3,3}{teal!50};
    \node at (1.8,-1.2) {\setlength{\arraycolsep}{1pt} \renewcommand{\arraystretch}{1.2}$\begin{array}{rcl} X&=& \{2\} \\ \overline{b} &=& (\textcolor{violet}{1},\textcolor{teal}{3}) \end{array}$};
    
\end{tikzpicture}
\hspace{.6 in}
\begin{tikzpicture}[scale=0.6]
    \node at (0,2.5) {\underline{Construction B:} $X \in \binom{[\row{n}+1]}{\ind{n}-1}=\binom{[3]}{0}$};
    \tri[(0,0)]{0,2,0,0,0}{teal!50};
    \tri[(0,0)]{0,0,0,4,4}{blue!50};
    \node at (-.2,-1.2) {\setlength{\arraycolsep}{1pt} \renewcommand{\arraystretch}{1.2}$\begin{array}{rcl} X&=& \emptyset \\ \overline{b} &=& (0,\textcolor{teal}{1},\textcolor{blue}{2}) \end{array}$};
    
\end{tikzpicture}
    \caption{Constructing the three maximum weight jump-float sequences in $\maxweightset{5}$ with Theorem~\ref{thrm:structure}.}
    \label{fig:construction_example1}
\end{figure}

% We provide a secondary, more visual interpretation of this theorem. Consider the perfect sequences (and their triangle interpretations) that bound row $k$, namely sequences of lengths $n'=\T_{k+1}-2$ and $n''=\T_{k+2}-2$. By Theorem~\ref{thrm:unique}, we know the block sequences of these sequences are $b'=(1,2,\ldots k)$ and $b''=(1,2,\ldots,k+1)$. We can consider the constructions given in Theorem~\ref{thrm:structure} as modifying the blocks of the triangle interpretation, either adding columns to existing blocks of the perfect triangle of size $n$ or removing columns from existing blocks of the triangle of size $n'$. 

% Consider a length of sequences $j$ that is in row $k$, that is $\row{j}=k$. Construction $A$ adds $\ind{j}$ columns to the triangle of size $n'$, while Construction B removes $\row{j}+1-(\ind{j}-1)$ columns from the triangle of size $n''$. Straightforward algebra shows this yields a triangle (and therefore a jump-float sequence) of size $j$.

%% file: conclusion.tex
\section{Conclusion}

\label{sec:conclusion}

We have characterized the $\pattern$-avoiding permutations with maximum inversion number. We accomplished this by focusing on the Lehmer codes for these permutations, which capture the inversion structure. We provided an algebraic criterion for these Lehmer codes, and named this restricted set of codes as jump-float sequences. By further restricting our attention to the maximizing permutations, we narrowed our focus to jump-float sequences that could be represented using so-called block sequences. Our proof of Theorem \ref{thrm:correct} provides a recursive construction of the block sequences for the maximizing permutations. We also provide a direct construction of these block sequences in Theorem \ref{thrm:structure}. Of course, these block sequences can be used to create the permutations themselves.

In Corollary \ref{cor:correct-perm}, we observed that our main theorem immediately gives analogous results for $\underline{12}3$-avoiding permutations, $3\underline{21}$-avoiding permutations and $1\underline{23}$-avoiding permutations. As such, what do the Lehmer codes for these families look like?
We can immediately answer this question for $\underline{12}3$-avoiding permutations. This is because $\pi \in \permset{n}(\pattern)$ if and only if its complement $\pi^c \in \permset{n}(\underline{12}3)$. So there is an an elementary bijection from Lehmer codes for $\permset{n}(\pattern)$ to Lehmer codes for $\permset{n}(\underline{12}3)$, namely $(\ell_1, \ldots, \ell_n) \mapsto (\ell_1', \ldots, \ell_n')$ where $\ell_k' = n-k - \ell_k$. Therefore our explicit construction for $\permset{n}(\pattern)$ also leads to an explicit construction for $\permset{n}(\underline{12}3)$.

Meanwhile, we have $\pi \in \permset{n}(\pattern)$ if and only if $\pi^r \in \permset{n}(1\underline{23})$ and $\pi^{r.c} \in \permset{n}(3\underline{21})$. There is not an elementary algebraic map between the Lehmer codes for permutations related by reversal. So it is an open question to characterize the structure of the Lehmer codes for permutations that avoid these vincular patterns, and their extremal elements in the corresponding posets. 

Finally, Section \ref{sec:oeis_ver} contains an algebraic verification that $ \{ \maxinv \permset{n+1}(\pattern)  \}_{n \geq 1}$ matches OEIS sequence A023536 \cite{oeis}. We are curious if there is a combinatorial proof of this equivalence.